\documentclass{amsproc} 

\usepackage{cite}
\usepackage{epsfig}
\usepackage{amsfonts}
\usepackage{amssymb}
\usepackage{amsmath}
\usepackage{amsthm}
\usepackage{latexsym}
\usepackage{mathrsfs}
\usepackage{enumerate}
\usepackage{array}
\usepackage{graphics}
\usepackage[vcentermath,enableskew]{youngtab}

\newtheorem{thm}{Theorem}[section]
\newtheorem{lem}[thm]{Lemma}

\theoremstyle{definition}

\theoremstyle{remark}
\newtheorem{rmk}[thm]{Remark}

\DeclareMathOperator{\trace}{tr}
\DeclareMathOperator{\height}{ht}

\DeclareMathOperator{\res}{Res}
\DeclareMathOperator{\ind}{Ind}

\DeclareMathOperator{\GL}{GL}
\DeclareMathOperator{\SL}{SL}

\newcommand{\C}{\mathbb{C}}
\newcommand{\R}{\mathbb{R}}
\newcommand{\Z}{\mathbb{Z}}

\newcommand{\Q}{\mathbb{Q}}
\newcommand{\Young}{\mathcal{Y}}
\newcommand{\w}{\mathsf{w}}
\newcommand{\p}{\mathbf{p}}
\newcommand{\s}{\mathbf{s}}

\title{An introduction to the half-infinite wedge}
\author{Rodolfo R\'ios-Zertuche}

\begin{document}

\begin{abstract}
After a quick review of the representation theory of the symmetric group, we give an exposition of the tools brought about by the so-called half-infinite wedge representation of the infinite symmetric group. We show how these can be applied to find the limit shapes of several distributions on partitions. We also briefly review the variational methods available to compute these limit shapes.

\end{abstract}
\maketitle
\tableofcontents
\section{Introduction}
\label{sec:intro}
The representation theory of the infinite symmetric group is located at the crossroads of a number of subjects that, together, give light to its nature and permit its application for the solution of a wide range of problems. These subjects include combinatorics, probability, arithmetic, algebraic geometry, and variational analysis. In this review article, we have tried to compile the main facts and techniques developed in relation with this theory, with the aim of reducing the steepness of the learning curve for these useful tools.

Although for many parts of the theory there exist very good introductions, such as \cite{macdonald, sagan,stanley}, there are important tools for which the standard references are still quite hard to drudge through. A good example is the half-infinite wedge fermionic Fock space, for which the standard reference \cite[Chapter 14]{kac} is admittedly inadequate for the beginner who does not want to become a specialist in a much broader theory. Other good sources, like \cite{miwajimbodate} or \cite[Section 2]{GWH1}, do not go into enough depth because they pursue different aims. Here we try to explain the tool and its techniques, and we also give examples of its application.

We feel that the increased production of research papers that leverage the usefulness of these tools (see for example \cite{pillow,branchedcoverings,GWH1,daveshokounkov,smirnov,okounkovreshetikhin,matsumoto}) justifies this new effort for their exposition.

We start with a quick review of the theory, including general facts of the representations of finite groups applied to the symmetric group in Section \ref{sec:reptheory}.

In Section \ref{sec:dimensionformulas}, we review the different available formulas for the dimension of the irreducible representations, and what is known about their asymptotics. We also include some words on the useful Okounkov-Olshanski formula for the dimension of skew Young diagrams.

Then we explain the computation of the characters of the symmetric group through the Murnaghan-Nakayama rule (Section \ref{sec:murnak}) and through the theory of symmetric functions (Section \ref{sec:symfuncs}).

All comes into play in the half-infinite wedge fermionic Fock space, which we explain in Section \ref{sec:fockspace}. Because they are so useful, we have included a few words on computations with vertex operators in Section \ref{sec:traces}. We also give a quick review of the theory of the simplest theta functions, which are quite prone to appear in this type of calculation.

Because it is useful for the examples we want to give, we also talk about the basic variational method to deal with the asymptotics of products and quotients of hooks. This is the content of Section \ref{sec:variational}.

Finally, we give several examples in Section \ref{sec:probability}. These concentrate on the issue of finding limit shapes of different distributions.

The reader who is in a hurry to learn how to use the half-infinite wedge fermionic Fock space can skip Sections \ref{sec:branching},
\ref{sec:dimensionformulas}, \ref{sec:theta}, and \ref{sec:variational}, and is encouraged to work out the example in Section \ref{sec:uniform} carefully.

Part of the material of this article was originally written for the author's thesis \cite{mythesis}. It has been significantly expanded, carefully revised, and published in the hope that others will find it useful. Needless to say, this introduction is by no means exhaustive in any direction.

\subsubsection*{Acknowledgements.}
The author is deeply indebted to his adviser, Professor Andrei Okounkov, who taught him virtually all he knows in relation to the subject of this paper. He is also very grateful to Dr. No\'e B\'arcenas Torres, Dr. Fernando Galaz Garc\'ia, and Dr. M\'onica Moreno Rocha, organizers of the workshop ``Taller de vinculaci\'on: matem\'aticos mexicanos j\'ovenes en el mundo,'' for a wonderful conference, for their encouragement to get this piece written up, and for making its publication possible. The author is also very grateful for the comments of the anonymous referees that have helped to greatly improve this paper.

\section{General representation theory of the symmetric group}
\label{sec:reptheory}
The \emph{symmetric group} $S(n)$ in $n$ elements is the set of bijections of the set of $n$ elements $\{1,2,\dots,n\}$. Its elements are known as \emph{permutations}. The group operation is the composition of permutations.

For simplicity, we will work throughout over the field of complex numbers, which we denote $\C$, and we will implicitly assume that all vector spaces are defined over this field.

A \emph{representation} $(V,\rho)$ of the symmetric group $S(n)$ consists of a vector space $V$ and a homomorphism $\rho:S(n)\to \GL(V)$ of $S(n)$ into the multiplicative group $\GL(V)$ of invertible linear transformations on $V$. A representation is \emph{reducible} if there are nontrivial subspaces $U,W\subset V$ such that $V= U\oplus W$ and $\rho$ factors through $\GL(U)\times \GL(W)$, that is, there exist representations $\rho_U:S(n)\to \GL(U)$ and $\rho_W:S(n)\to \GL(W)$ such that $\rho(\sigma)|_U=\rho_U(\sigma)$ and $\rho(\sigma)|_W=\rho_W(\sigma)$ for all $\sigma \in S(n)$. If two such subspaces fail to exist, we say that $(V,\rho)$ is an \emph{irreducible representation} of $S(n)$.

The \emph{dimension} of the representation $(V,\rho)$ is simply the dimension of the underlying vector space $V$.

Two representations $(U,\rho_U)$ and $(V,\rho_V)$ are \emph{isomorphic} if there exists a linear bijection $\phi:U\to V$ such that $\phi\circ\rho_V=\rho_U$. The existence of an isomorphism of representations implies that the underlying vector spaces $U$ and $V$ are of the same dimension.

For an arbitrary matrix $A$, we denote its trace by $\trace A$.
The \emph{character} $\chi_V:S(n)\to \C$ of a representation $(V,\rho)$ is the function $\chi_V(\sigma)=\trace\rho(\sigma)$. The character is invariant under conjugation: for every $\sigma, \tau\in S(n)$,
\[\chi_V(\tau^{-1}\sigma\tau)=\trace\rho(\tau^{-1}\sigma\tau)=\trace\left(\rho(\tau)^{-1}\rho(\sigma)\rho(\tau)\right)=\trace \rho(\sigma)=\chi_V(\sigma).\]
It is thus constant over conjugacy classes.

The classification of the finite-dimensional representations of the symmetric group can be easily achieved using the standard theory of group characters; for details we refer the reader to \cite{fultonharris,sagan}, for example. 

In order to explain this classification, we must introduce some combinatorial objects.
A \emph{partition} of a positive integer $n$ is a way to write $n$ as a sum $n=a_1+a_2+\cdots+a_k$ of positive integers $a_1,a_2,\dots,a_k$, and is usually denoted $(a_1,a_2,\dots,a_k)$. Conventionally, we order the parts so that $a_1\geq a_2\geq\dots\geq a_k$. For example, the partitions of the number 5 are $(5)$, $(4,1)$, $(3,2)$, $(3,1,1)$, $(2,2,1)$, $(2,1,1,1)$, and $(1,1,1,1,1)$. It is also standard to abbreviate repetitions using exponents. For example, $(3,3,1,1,1,1,1)$ can be written $(3^2,1^5)$. We will use the notation $|\lambda|=\sum_i\lambda_i=n$, and $\ell(\lambda)$ will denote the number of parts of $\lambda$.

The irreducible representations of the symmetric group $S(n)$ can be indexed by partitions of $n$, that is, for each partition $\lambda$ of $n$ there exists an irreducible representation $(V_\lambda,\rho_\lambda)$ of $S(n)$. We will denote by $\chi^\lambda$ the character of the representation corresponding to the partition $\lambda$.

In the case of the symmetric group $S(n)$, the conjugacy classes are also indexed by partitions: the lengths of the cycles of an element of $S(n)$ are invariant under conjugation, they completely determine the conjugacy class, and they correspond to a partition of $n$. For example, the permutation in $S(5)$ that takes $1\mapsto 2$, $2\mapsto 5$, $5\mapsto 1$, $3\mapsto 4$, $4\mapsto 3$ is denoted in cycle notation by $(1,2,5)(3,4)$; it is composed of two cycles, $(1,2,5)$ and $(3,4)$, of lengths 3 and 2 respectively. Its conjugacy type corresponds to the partition $(3,2)$, and contains all other permutations with one cycle of length 3 and one of length 2, such as $(1,2,3)(4,5)$ and $(3,4,5)(1,2)$. The partition $\lambda$ that determines the conjugacy class $C_\lambda$ of a permutation is also known as its \emph{cycle type}. The number of elements $|C_\lambda|$ in a conjugacy class is equal to $|S(n)|/\mathfrak z(\lambda)$, where
\[\mathfrak z(\lambda)=\prod_{i=1}^\infty i^{m_i(\lambda)} m_i(\lambda)!\]
is the size of the centralizer of $C_\lambda$, and $m_i(\lambda)$ is the number of parts in $\lambda$ that are equal to $i$.

The only permutation of cycle type $(1,1,\dots,1)$ is the identity. A permutation $\sigma$ of cycle type $(2,2,\dots,2,1,1,\dots,1)$ is an involution, that is, $\sigma\circ\sigma$ is the identity. A permutation of cycle type $(2,2,\dots,2)$ is an involution without fixed points.

Since the characters are constant over conjugacy classes, we usually denote the character of a permutation of cycle type $\mu$ in the irreducible representation corresponding to the partition $\lambda$ by $\chi^\lambda(\mu)$. In the following sections, we study some related objects that provide ways to compute the values of these functions more or less explicitly, and that relate them to other combinatorial constructions.

Because of their connection with the classification of irreducible representations and of conjugacy classes, partitions play a crucial role in the theory. We remark that to a partition $\lambda=(\lambda_1,\lambda_2,\dots,\lambda_k)$, where $\lambda_1\geq\lambda_2\geq\cdots\geq\lambda_k\geq0$, we can assign a \emph{Young diagram} consisting of a drawing of rows of boxes stacked on each other, with row $i$ consisting of of $\lambda_i$ square boxes for $i=1,2,\dots,k$. For example, if our partition is $(5,4,4,2)$, the diagram looks like this:
\[\yng(5,4,4,2)\]
In this diagram, each square is called a \emph{cell}.

We recall two more general facts about group representations. The first one is the decomposition of the group algebra $\C G$ of any finite group $G$, which is itself a representation of $G$. The group algebra $\C G$ is the algebra obtained by forming the vector space over $\C$ whose generators are the elements of $G$ and whose product operation is inherited from the operation of $G$. It turns out to contain each of the irreducible representations of $G$ exactly as many times as their dimensions, that is,
\[\C G=\bigoplus_i V_i^{\oplus\dim V_i},\]
where the subspaces $V_i$ range over the different irreducible representations of $G$. See, for example, \cite[Section 1.10]{sagan}.

The second one is \emph{Schur's lemma}: Let $V$ and $W$ be two irreducible representations of $G$. If $\phi:V\to W$ is a $G$-homomorphism (i.e., $\phi$ commutes with the action of $G$), then either $\phi$ is a $G$-isomorphism, or $\phi$ is the zero map. See, for example, \cite[Section 1.6]{sagan}.

\subsection{The branching rule and Frobenius reciprocity}
\label{sec:branching}
The restriction of a representation $\rho:S(n)\to \GL(V)$ of $S(n)$ to a smaller symmetric group $S(m)$, $m<n$, induces a restriction of the corresponding representation, which we denote $\res_{S(m)} V$. If $\rho$ is irreducible and corresponds, say, to the partition $\lambda$ of $n$, $\res_{S(m)} V$ is, in general, reducible. The \emph{branching rule} gives its decomposition into irreducible representations. It states that, if $n=m+1$ and $V^\lambda$ is the irreducible representation corresponding to the partition $\lambda$, then
\[\res_{S(m)}V^\lambda=\bigoplus_{\lambda = \mu\cup\square}V^{\mu}, \]
where the direct sum is taken over all partitions $\mu$ of $m$ from which $\lambda$ can be obtained by adjoining one cell $\square$.
For example,
\[
\res_{S(11)}V^{(4,3,3,2)}=V^{(4,3,3,1)}\oplus V^{(4,3,2,2)} \oplus V^{(3,3,3,2)}.
\]
Note that, iterating, this gives a complete description of the decomposition corresponding to any pair $m<n$.
For a proof, see for example \cite[Section 2.8]{sagan} or \cite[Section 7.3]{fulton}. If $\chi$ is the character of $\rho$, then we will denote by $\res_{S(m)} \chi$ the character of $\res_{S(m)}V$. 

A representation $\rho:S(m):\to \GL(V)$ induces a representation $\ind_{S(n)} V$ of $S(n)$ (see \cite{sagan} for a full discussion).
There is a version of the branching rule in this case:
\[\ind_{S(n)}V^\lambda=\bigoplus_{\lambda \cup \square= \mu}V^{\mu}, \]
where the sum is taken over all partitions $\mu$ that can be formed by adjoining a cell $\square$ to $\lambda$. Let $\chi$ be the character of $\rho$ and denote by $\ind_{S(n)} \chi$ the character of the induced representation. There is a way to compute $\ind_{S(n)} \chi$ using Frobenius reciprocity, described next.

The vector space of real-valued functions on $S(n)$ contains a subspace generated by the characters on $S(n)$. This subspace has an inner product
\[\langle \chi_1,\chi_2\rangle_{S(n)}=\frac 1{n!} \sum_{\sigma\in S(n)} \chi_1(\sigma)\chi_2(\sigma).\]
The characters of irreducible representations turn out to be orthonormal with respect to this inner product.
The relation
\[\langle \ind_{S(n)}\chi_1,\chi_2\rangle_{S(n)}=\langle \chi_1,\res_{S(m)}\chi_2\rangle_{S(m)},\]
which is true for all characters $\chi_1$ and $\chi_2$ of representations of $S(m)$ and $S(n)$, $m<n$, respectively, is known as \emph{Frobenius reciprocity}. For a proof see \cite[Section 3.3]{fultonharris}, \cite[Theorem 1.12.6]{sagan}.

\section{Formulas for the dimension}
\label{sec:dimensionformulas}
\subsection{The hook formula}
\label{sec:hookformula}
The formula obtained in 1954 by J. Frame, G. Robinson, and R. Thrall \cite{framerobinsonthrall} to calculate the dimension of the irreducible representation of $S(n)$ corresponding to the partition $\lambda$, or equivalently, to calculate the character $\chi^\lambda(1,1,\dots,1)$, has become known as the \emph{hook formula} due to the combinatorial objects it involves.

To each cell in a Young diagram we can assign several quantities. For our purposes, the most important one is the \emph{hook length}, which we now define. For a given cell, the \emph{hook} consists of the cell itself, together with all the cells to the right of it and all  the cells below it. So, for example, in this diagram we have marked with bullets the hook of cell $(2,2)$:
\[\young(\hfil\hfil\hfil\hfil\hfil,\hfil\bullet\bullet\bullet,\hfil\bullet\hfil\hfil,\hfil\bullet)\]
The \emph{hook length} is the number of cells in the hook. In the example of the last diagram, the hook length is 5. In the following diagram we have filled in each cell with the corresponding hook length:
\[\young(87541,6532,5421,21)\]

Now let $\lambda$ be a partition of $n$, and let $h(\lambda)$ be the product of the hook lengths of $\lambda$. For the example above, $\lambda=(5,4,4,2)$, $n=15$, and $h(\lambda)=16\,128\,000$. Then the hook formula, discovered by Frame, Robinson, and Thrall, says that
\begin{equation}\label{eq:hookformula}
\boxed{\dim\lambda:=\chi^\lambda(1,1,\dots,1)=\frac{n!}{h(\lambda)}.}
\end{equation}
Here, $(1,1,\dots,1)$ stands for the partition with $n$ parts equal to 1 and corresponds to the conjugacy class of the identity.

In the example above, formula \eqref{eq:hookformula} gives
\[\dim(5,4,4,2)=\chi^{(5,4,4,2)}(1,1,\dots,1)=\frac{15!}{h((5,4,4,2))} =81\,081.\]
A proof of the hook formula can be found in the original paper \cite{framerobinsonthrall}, and also in \cite[Examples I.1.1 and I.7.6]{macdonald}, \cite[Section 3.10]{sagan}, and \cite[Corollary 7.21.6]{stanley}.

\subsection{Other formulas}\label{sec:otherformulas}
There is an alternative way to express this as a function of the coordinates of the partition $\lambda$ \cite[Example I.1.1]{macdonald}:
\begin{equation}\label{eq:hookalternative}
\boxed{
\dim\lambda=\frac{|\lambda|!\prod_{i<j}(\lambda_i-\lambda_j-i+j)} {\prod_i(\lambda_i+\ell(\lambda)-i)!}.}
\end{equation}
Here $\ell(\lambda)$ denotes the number of parts of $\lambda$.

Yet another formula can be obtained if one uses instead the so-called Frobenius coordinates for a partition. For each partition $\lambda$, there is a \emph{transposed partition} $\lambda'$ that is obtained by exchanging rows by columns in the Young diagram; for example, if $\lambda=(5,4,4,2)$, then $\lambda'=(4,4,3,3,1)$:
 \[\lambda=\yng(5,4,4,2)\qquad\lambda'=\yng(4,4,3,3,1)\]
  The \emph{Frobenius coordinates} are defined as follows: for a partition $\lambda=(\lambda_1,\lambda_2,\dots,\lambda_k)$, let $d$ be the largest integer $\leq k$ such that $\lambda_d\geq d$. In other words, $d$ is the size of the diagonal of $\lambda$. Then let $p_i=\lambda_i-i$ and $q_i=\lambda'_i-i$ for $i=1,2,\dots,d$, where the numbers $\lambda'_i$ are the parts of the transposed partition $\lambda'$. Of course, $p_i,q_i\geq0$. For $\lambda=(5,4,4,2)$ as above, the Frobenius coordinates are $p=(4,2,1)$, $q=(3,2,0)$. We have \cite[eq. (2.7)]{olshanski}
\[\boxed{\dim\lambda=\frac{\prod_{i<j}(p_i-p_j)(q_i-q_j)}{\prod_{i,j}(p_i+q_j+1)
\prod_i p_i!\,q_i!}.}\]
This can also be written as a determinant \cite[eq. (2.2)]{borodinokounkovolshanski}:
\[\frac{\dim\lambda}{|\lambda|!}=\det\left[\frac{1}{(p_i+q_j+1)p_i!\,q_i!}\right]_{1 \leq i,j\leq d}.\]

\subsection{Combinatorial interpretation}
\label{sec:combinatorialdim}

Combinatorially, $\dim\lambda$ is the number of ways in which one can construct the Young diagram of $\lambda$ by adding one cell at a time, starting from the upper left corner and making sure that, at each step, the diagram obtained corresponds to some partition. So for example, for the partition $(3,2)$, one could add the cells in the following orders:
\[\young(123,45)\quad
\young(124,35)\quad
\young(125,34)\quad
\young(134,25)\quad
\young(135,24)\quad
\]
In each row and column the numbers must be increasing. We conclude that $\dim(3,2)=5$. This interpretation follows immediately from the branching rule (see Section \ref{sec:branching}) or from the Murnaghan-Nakayama rule (see Section \ref{sec:murnak}).

\subsection{Asymptotic formula}

Vershik and Kerov \cite{vershikkerov1985} obtained noteworthy asymptotic estimates for the maximal dimension of the irreducible representation corresponding to the partition $\lambda$ as $|\lambda|\to\infty$.

They found that there exist positive constants $c$ and $c'$ such that
\[\sqrt{n!}\,\exp\left(-\frac c2\sqrt n\right)\leq\max_{\lambda\in S(n)}\dim\lambda\
\leq\sqrt{n!}\,\exp\left(-\frac {c'}2\sqrt n\right).\]
Since the actual values of $c$ and $c'$ are not known explicitly, the true value of this formula lies on the description it gives of the asymptotic situation as $n\to\infty$.

It is also possible to obtain asymptotics for the typical dimensions for some probability distributions; see Section \ref{sec:probability} for examples.

\subsection{The Okounkov-Olshanski formula}
\label{sec:okolshformula}
In this section we will state a formula due to A. Okounkov and G. Olshanski \cite{okounkovolshanski} for the dimension $\dim (\lambda/\mu)$ of a skew diagram.

Let $\lambda$ and $\mu$ be two partitions, and assume that $\mu_i\leq\lambda_i$, $i=1,2,\dots$ A \emph{skew Young diagram} $\lambda/\mu$, consists of the cells of the Young diagram of $\lambda$ that would not belong to the Young diagram of $\mu$ if we were to overlap them. For example, let $\mu=(3,1)$ and $\lambda=(5,4,4,2)$. Then
\[\lambda/\mu=\young(:::\hfil\hfil,:\hfil\hfil\hfil,\hfil\hfil\hfil\hfil,\hfil\hfil).\]

The \emph{dimension} of $\lambda/\mu$ is defined as a straight-forward generalization of the combinatorial principle explained in Section \ref{sec:combinatorialdim}: it is the number of ways in which one can assign the numbers $1,2,\dots, |\lambda|-|\mu|$ to the cells $\lambda/\mu$ in such a way that in each row the numbers increase as one moves downwards, and in each column the numbers increase as one moves to the right.

To simplify notations, let
\[(x\downharpoonright k)=x(x-1)(x-2)\cdots(x-k+1).\]
This is known as the \emph{Pochhammer symbol} and another common notation for it is $(x)_k$.

We define the \emph{shifted Schur polynomials} in $n$ variables, indexed by the partition $\mu$, as the following ratio of two $n\times n$ determinants:
\[\s_\mu(x_1,\dots,x_n)=\frac{\det [(x_i+n-i\downharpoonright \mu_j+n-j)]} {\det[(x_i+n-i\downharpoonright n-j)]},\quad 1\leq i,j\leq n.\]
These polynomials satisfy \cite{okounkovolshanski} a stability condition
\[\s_\mu(x_1,\dots,x_n,0)=\s_\mu(x_1,\dots,x_n),\]
which allows us to take inverse limits, just as in the definition of symmetric functions (see Section \ref{sec:symfuncs}). The resulting objects are known as \emph{shifted Schur functions} and we will denote them by $\s_\mu(x_1,x_2,\dots)$. These are also sometimes called \emph{Frobenius Schur functions} (see for example \cite{frobeniusschur}). The algebra $\Lambda^*$ they generate coincides with the algebra generated by the shifted-symmetric power functions $\p_k$.

We then have \cite{okounkovolshanski}:
\begin{equation}\label{eq:dimensionofskewdiagram}
\boxed{\dim(\lambda/\mu)=\frac{\s_\mu(\lambda)\,\dim\lambda}{ (|\lambda|\downharpoonright|\mu|)}.}
\end{equation}

\section{The Murnaghan-Nakayama rule}
\label{sec:murnak}
The Murnaghan-Nakayama rule gives a combinatorial algorithm to compute the value of the character $\chi^\lambda(\mu)$ of an irreducible representation of the symmetric group $S(n)$ corresponding to the partition $\lambda$ and evaluated at an element of cycle type $\mu$.

In order to state the rule, we need to define some combinatorial objects. Recall the definition of a skew Young diagram from Section \ref{sec:okolshformula}.
A special kind of Young diagram that we will be interested in is the \emph{strip}, which is defined to be a connected skew Young diagram that does not contain a $2\times2$ block:
\[\yng(2,2)\]
A skew Young diagram is connected if we can get from any cell to any other cell jumping on \emph{adjacent} cells. For example, the following are strips:
\begin{equation}\label{strips}
\young(::\hfil,::\hfil,::\hfil,:\hfil\hfil,:\hfil,\hfil\hfil) \qquad\young(\hfil\hfil\hfil\hfil,\hfil,\hfil,\hfil)\qquad \young(\hfil,\hfil,\hfil,\hfil,\hfil)
\end{equation}
The following are \emph{not} strips:
\[\young(::\hfil,::\hfil,::\hfil,:\hfil,:\hfil,\hfil\hfil)\qquad \young(\hfil,\hfil\hfil,:\hfil,\hfil\hfil)\qquad\young(:::::\hfil,::\hfil\hfil\hfil\hfil,:\hfil\hfil,\hfil\hfil\hfil,\hfil)\]
Here, the first example is disconnected (it is not enough to have a common corner for cells to be adjacent), the second example cannot be obtained as a difference of Young diagrams, and the third contains a $2\times2$ block.

Now suppose that the partitions $\mu$ and $\lambda$ are both partitions of the same number, that is, $\sum \lambda_i=\sum \mu_i$. A \emph{$\mu$-strip decomposition} of the partition $\lambda$ is a way to recover $\lambda$ by successively adjoining strips of size $\mu_1,\mu_2,\dots$. More precisely, it is a sequence of partitions $\eta^1,\eta^2,\dots,\eta^k$, where $k$ is the number of parts in $\mu$, such that $\eta^1$ is a strip of size $\mu_1$, $\eta^i/\eta^{i-1}$ is a strip of length $\mu_i$, $i=2,\dots, k$, and $\eta^k=\lambda$. For example, if $\lambda=(5,4,4,2)$ and $\mu=(5,3,2,2,2,1)$, a $\mu$-strip decomposition of $\lambda$ can be represented as follows: we will use number 1 to represent the strip of size 5, then number 2 to represent the strip of size 3, and so on, always assigning the number $i$ to the strip of length $\mu_i$:
\begin{equation}\label{decompositions}
\young(11136,1223,1255,44)\qquad\young(11233,1225,1445,16)\qquad\young(11244,1225,1335,16)
\end{equation}
Note that the second and third examples are different only in the order of the third and fourth strips.

The \emph{height} of a strip $\lambda/\mu$ is equal to the \emph{vertical} distance between the center of a cell in the lowest row of the strip to the center of a cell in the highest row of the strip, assuming that each cell has side length 1. For example, the strips in \eqref{strips} have heights 5, 3, and 4, respectively. The \emph{height of a $\mu$-strip decomposition} equals the sum of the heights of the involved strips. For example, the height of the first decomposition in \eqref{decompositions} is $2+1+1+0+0+0=4$, and the heights of the other two are both equal to $3+1+0+0+1+0=5$.

The \emph{Murnaghan Nakayama rule} is the simple assertion that
\begin{equation}\label{murnakruleeq}
\boxed{\chi^\lambda(\mu)=\sum_{\{\eta_i\}} (-1)^{\height \{\eta^i\}},}
\end{equation}
where the sum is taken over all the $\mu$-strip decompositions $\{\eta^i\}$ of $\lambda$. For example, if $\lambda=(4,3,2)$ and $\mu=(6,2,1)$, then there are exactly two $\mu$-strip decompositions of $\lambda$, namely (in the same notations as above),
\[\young(1111,123,12)\qquad\textrm{and}\qquad\young(1111,122,13)\]
The heights of these decompositions are $2+1+0=3$ and $2+0+0=2$, respectively, so the Murnaghan-Nakayama rule implies that
\[\chi^{(4,3,2)}((6,2,1))=(-1)^3+(-1)^2=0.\]

Proofs of the Murnaghan-Nakayama rule can be found in \cite[Exercises I.3.11 and I.7.5]{macdonald}, \cite[Section 4.10]{sagan}, or \cite[Theorem 7.17.3]{stanley}.

\section{Symmetric and shifted-symmetric functions}
\label{sec:symfuncs}
 \subsection{Schur functions and the characters of the symmetric group}
 \label{sec:normalschurskewschur}
 
 For variables $x_1,x_2,\dots,x_n$, a \emph{symmetric polynomial} is a polynomial $p(x_1,x_2,\dots\,x_n)$ with coefficients in the rational numbers $\Q$ such that $p$ is invariant by any reordering of the variables:
\[p(x_1,x_2,\dots,x_n)=p(x_{\sigma(1)},x_{\sigma(2)},\dots,x_{\sigma(n)})\]
for all $\sigma\in S(n)$. Denote by $P_n$ the set of all symmetric polynomials in $n$ variables.

Let $\psi:P_{n+1}\to P_n$ be the specialization to $n$ variables, given by
\[\psi(p)(x_1,x_2,\dots,x_{n+1})=p(x_1,x_2,\dots,x_n,0).\]
Using the identification provided by $\psi$, we can form the inverse limit
\[\Lambda=\underleftarrow \lim \,P_n\]
whose elements are known as \emph{symmetric functions}, and consist of formal sums of monomials of the form
\[a x_{i_1}x_{i_2}\cdots x_{i_k}\]
in countably many variables $x_1,x_2,\dots$, for $a\in \Q$, $i_1,\dots,i_k\in \mathbb N$, where the numbers $i_j$ may be repeated. By construction, symmetric functions are invariant by any finite permutation. They form a ring, denoted $\Lambda$, that is obviously also a vector space over $\Q$.

A basis of the symmetric functions $\Lambda$ is given by the \emph{symmetric power functions} $p_\mu$. To define them, we let
\[p_k(x_1,x_2,\dots)=x_1^k+x_2^k+\cdots,\quad k=1,2,\dots,\]
and then extend this definition to partitions $\mu=(\mu_1,\mu_2,\dots,\mu_k)$ by letting
\[p_\mu=p_{\mu_1}p_{\mu_2}\cdots p_{\mu_k}.\]
In other words, $\Lambda$ is algebraically generated by the functions $p_k$ \cite{macdonald}.

While the symmetric power functions are by far the simplest linear basis of the symmetric functions $\Lambda$, there are a few other bases of $\Lambda$ that are of interest. Among these is the one composed of the so-called \emph{Schur functions} $s_1,s_2,\dots$, which can be defined in many equivalent ways. One way is as the sum \cite{macdonald}
\begin{equation}\label{powertoschur}
\boxed{
s_\lambda=\sum_\rho \frac{\chi^\lambda(\rho)}{\mathfrak z(\rho)} p_\rho,
}
\end{equation}
where the sum is over all partitions $\rho=(\rho_1,\rho_2,\dots,\rho_k)$, $\mathfrak z(\rho)=\prod_i i^{m_i(\rho)}m_i(\rho)!$, $m_i(\rho)$ is the number of parts $\rho_j$ of $\rho$ that are equal to $i$, and $\chi^\lambda(\rho)$ stands for the character of the representation of the symmetric group corresponding to the partition $\lambda$ evaluated at an element of cycle type $\rho$.

A second way to define the Schur functions is by their characterization as being the symmetric functions that specialize (by setting the variables $x_{n+1},x_{n+2},\dots$ to zero) to the \emph{Schur polynomials}, defined by quotients of determinants:
\begin{equation}
s_\lambda(x_1,x_2,\dots,x_n,0,0,\dots)=\frac{\det(x_i^{\lambda_j+n-j})_{i,j=1,\dots,n}}{\det(x_i^{n-j})_{i,j=1,\dots,n}}.
\end{equation}
This form of the definition, together with equation \eqref{powertoschur} gives a way to compute the characters $\chi^\lambda(\rho)$. Other ways to define $s_\lambda$ are described in the following sections.

We want to mention an extremely important generalization that will come up in Section \ref{sec:vertexopvalue}. Endow the space of symmetric functions with the inner product induced by
\[(s_\lambda,s_\mu)=\delta_{\lambda\mu}.\]
 Define the \emph{skew Schur functions} by
\[(s_{\lambda/\mu},s_\nu)=(s_\lambda,s_\mu s_\nu).\]
In other words, if we have the expansion of the product in the basis given by the Schur functions $s_\lambda$,
\[s_\mu s_\nu=\sum_\lambda c_{\mu\nu}^\lambda s_\lambda,\]
then we define
\[s_{\lambda/\mu}=\sum_\nu c_{\mu\nu}^\lambda s_\nu.\]
The functions $s_{\lambda/\mu}$ vanish unless the Young diagram of $\mu$ is completely contained inside the Young diagram of $\lambda$ (see \cite[(I.5.7)]{macdonald}). The numbers $c^\lambda_{\mu\nu}$ have a nice combinatorial interpretation known as the \emph{Littlewood-Richardson rule} that is out of the scope of this exposition; see \cite[Section 4.9]{sagan}, \cite[Section I.9]{macdonald}.

Apart from the symmetric power functions and the Schur functions, there are other bases of the space of symmetric functions that appear frequently in the subject. The first one is the one one given by the \emph{elementary symmetric functions},
\[e_\lambda=\prod_ie_{\lambda_i},\quad e_r=\sum_{i_1<i_2<\cdots<i_r}x_{i_1}x_{i_2}\cdots x_{i_r},\quad r\in \Z_+.\]

There is also the basis given by the \emph{monomial symmetric functions},
\[m_\lambda=\sum_\alpha x_1^{\alpha_1}x_2^{\alpha_2}\cdots,\]
where the sum is over all $\alpha\in\Z^\infty$ such that $\alpha_i=0$ for all but finitely many $i$, and the $\alpha_i$ that do not vanish are a permutation of the parts $\lambda_i$ of the partition $\lambda$. These directly generate the space of symmetric functions.

One last basis we need to mention is the one given by the \emph{complete symmetric functions} $h_r$, given by
\[h_r=\sum_{|\lambda|=r}m_\lambda.\]

Many interesting facts are known about the interaction of these bases and of their generating functions; the reader is encouraged to look at \cite[Chapter I]{macdonald}. We mention only two relations that will come up in Section \ref{sec:vertexopvalue}:
\begin{equation}\label{eq:skewschurh}
s_{\lambda/\mu}=\det\left(h_{\lambda_i-\mu_j+j-i}\right),\quad\sum_{r=1}^\infty h_rz^r=\exp\sum_{k=1}^\infty\frac{p_kz^k}k.
\end{equation}
The reference for these is \cite[pages 25 and 70]{macdonald}.

\subsection{Shifted-symmetric functions}
\label{sec:shiftedsymmetric}

Parallel to the symmetric functions, we have the \emph{shifted-symmetric functions}, which are symmetric on the shifted variables $x_i-i+\frac12$. Section \ref{sec:fockspace} will clarify why these coordinates are very natural.

The shifted-symmetric Schur functions were already defined in Section \ref{sec:okolshformula}. The \emph{shifted symmetric power functions} are given by
\[\p_\lambda=\prod_i\p_{\lambda_i},\quad
\p_k(x)=\sum_{j=1}^\infty\left[\left(x_j-j+\tfrac12\right)^k-\left(-j
+\tfrac12\right)^k\right]+\left(1-\frac1{2^k}\right)\zeta(-k),\]
where $\zeta$ is the Riemann zeta function. The last summand appears due to the $zeta$-function regularization of the divergent sum $\sum_i(x_i-i+\tfrac12)^k$; see \cite[Remark 2.14]{branchedcoverings}.
It will be useful to record their exponential generating function \cite[eq. (0.18)]{blochokounkov}, \cite[eq. (2.8)]{branchedcoverings}
\begin{equation}\label{eq:generatingshiftedsympower}
\sum_k \frac{z^k}{k!}\p_k(x)=\sum_ie^{\left(x_i-i+\frac{1}{2}\right)z}.
\end{equation}

\subsubsection*{Moments of the limit shape.}

Let $\Omega:\R_+\to\R_+$ be a non-increasing continuous function such that
\[\int_{\R_+}\Omega(x)\,dx=1.\]
Let $\lambda^n$, $n=1,2,\dots$, be a sequence of partitions such that $|\lambda^n|\to\infty $ as $n\to\infty$.
We say that $\lambda^n$ \emph{converges} to $\Omega$ if
\[\limsup_{n\to\infty}\max_j\left|\frac1{\sqrt n}\,\lambda^n_j-\Omega\!\left(\frac j{\sqrt n}\right)\right|=0.\]

Let $L_\infty:\R\to\R$ be the function whose graph corresponds to the graph of $\Omega$ rotated $45^\circ$ in the counter-clockwise direction (or $135^\circ$ if we think of the graph of $\Omega$ drawn upside down, resembling the lower-right contour of the Young diagram of a very large partition, rescaled to have area 1). Also, after similarly rotating the Young diagram of $\lambda^n$, let $L_n$ denote the function whose graph is the upper border of the rotated Young diagram of $\lambda^n$ rescaled by $1/\sqrt n$. See the diagrams on page \pageref{gr:contour}.

Then if $\lambda^n$ converge to $\Omega$, we also have that $L_n$ converges to $L_\infty$ in the topology of the supremum norm.

 We want to give an interpretation of the shifted power functions $\p_k(\lambda^n)$: they converge to the moments of the contour of the limit shape $L_\infty$. From the definition of $\p_k$ it is clear that as $n\to\infty$,
\[\left|n^{-\frac{k+1}2}\p_k(\lambda^n)
-\int_{\R_+}\left(\left(\Omega(x)-x\right)^k-(-x)^k\right)dx\right|\to0.\]
Now,
\begin{align}
\int_{\R_+}&\left(\left(\Omega(x)-x\right)^k-(-x)^k\right)dx
=\int_0^\infty\int_0^{\Omega(x)}\frac{d}{dy}(y-x)^k\,dy\,dx\notag\\
&=\int_0^\infty\int_0^{\Omega(x)}k(y-x)^{k-1}\,dy\,dx=
2\int_{-\infty}^\infty\int_{|s|}^{L_\infty(s)}ks^{k-1}\,dt\,ds\notag\\
&=2k\int_{-\infty}^\infty s^{k-1}\left(L_\infty(s)-|s|\right)\,ds
\label{eq:momentintegral}
\end{align}
where $s=y-x$ and $t=y+x$. The change of variables is done as per the following diagram, in which the curve $\Gamma$ depicts both the graph of $\Omega(x)$ and the graph of $L_\infty(s)$, and the integration domain is shaded in gray.

\begin{center}
\includegraphics{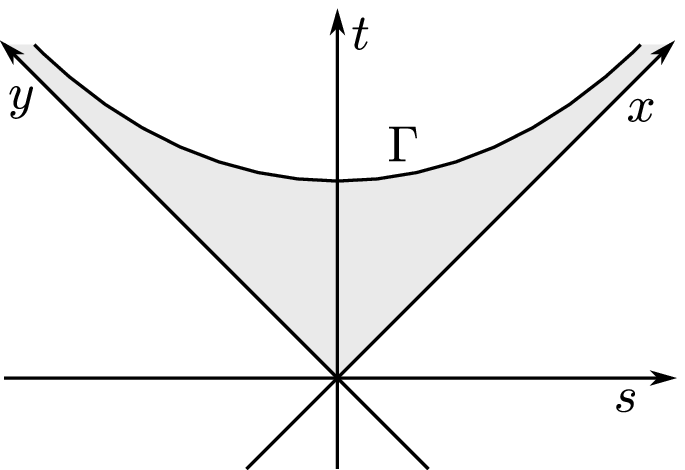}
\end{center}

In other words, if we consider $L_\infty(s)-|s|$ as a density with moments $\mu_n$,
\[\mu_n=\int_\R s^n(L_\infty(s)-|s|)\,ds,\]
then
\[\boxed{\left|\lambda^n\right|^{-\frac{k+1}{2}}\p_k(\lambda^n)\to 2k\mu_{k-1}\quad\textrm{as}\quad n\to\infty.}\]

Another derivation of the result of this section can be found in \cite{ivanovolshanski}.

\section{The half-infinite wedge fermionic Fock space}
\label{sec:fockspace}
In this section we describe an algebraic framework to work with partitions and their connections to the symmetric group. Other good sources for this material are \cite[Chapter 14]{kac}, \cite{miwajimbodate}, and \cite[Section 2]{GWH1}.

Let $\lambda=(\lambda_1,\lambda_2,\dots,\lambda_k)$ be a partition. For convenience, we identify this vector representation with an infinite sequence by extending it with zeros: \[\lambda=(\lambda_1,\lambda_2,\dots,\lambda_k,0,0,\dots),\]
so that $\lambda_i=0$ for $i>k$. Now let $\xi_i=\lambda_i-i+\frac12$ be the \emph{shifted coordinates} of $\lambda$. Obviously, for $i>k$, $\xi_i=-i+\frac12$.

We start with a countably-infinite-dimensional vector space $V$. We take a (countable) basis of $V$, and we denote its elements by $\dots,\underline {\frac52},\underline{\frac32},\underline{\frac12},\underline {-\frac12},\underline{-\frac32},\underline{-\frac52},\dots$; thus, this is a bidirectionally-infinite sequence of vectors that form a basis of $V$. The strange notation (using $\underline {\frac{2k+1}2}$ instead of something more normal like $u_k$) is well motivated by the geometric picture that we will now explain.

The \emph{0-charge sector of fermionic Fock space} $\Lambda_0^{\frac\infty2}V$ (also known as the \emph{half-infinite wedge} $\Lambda_0^{\frac{\infty}2}$) is a vector space generated by the basis of vectors $v_\lambda$ indexed by partitions $\lambda=(\lambda_1,\lambda_2,\dots)$. These are defined by:
\[v_\lambda=\underline{\xi_1}\wedge\underline{\xi_2}\wedge \underline{\xi_3}\wedge\cdots
=\underline{\lambda_1-\tfrac12}\wedge\underline{\lambda_2-\tfrac32} \wedge\underline{\lambda_3-\tfrac52}\wedge\cdots\wedge \underline{\lambda_i-i+\tfrac12}\wedge\cdots.\]
Here, $\lambda_1\geq\lambda_2\geq\lambda_3\geq\cdots\geq 0$ and the $\lambda_i$ eventually become zero. The symbol $\wedge$ indicates that adjacent factors are anti-commutative. For example, for $\lambda=(5,4,4,2)$, 
\begin{align*}
v_\lambda&=\underline{\tfrac92}\wedge\underline{\tfrac52} \wedge\underline{\tfrac32}\wedge\underline{-\tfrac32}\wedge \underline{-\tfrac92}\wedge
\underline{-\tfrac{11}2}\wedge\underline{-\tfrac{13}2}\wedge\cdots.\\
&=-\underline{\tfrac92}\wedge\underline{\tfrac32} \wedge\underline{\tfrac52}\wedge\underline{-\tfrac32}\wedge \underline{-\tfrac92}\wedge
\underline{-\tfrac{11}2}\wedge\underline{-\tfrac{13}2}\wedge\cdots,
\end{align*}
since the factors $\underline{\tfrac32}$ and $\underline{\tfrac52}$ have been exchanged.

The relation between $v_\lambda$ and $\lambda$ can be best visualized as follows.
We rotate the picture of the Young diagram of $\lambda$ by $135^\circ$ counter-clockwise, and we rescale it by $\sqrt2$. We then place black pebbles at each point $\xi_i$ on the reversed $x$-axis, and white ones in the remaining half-integers. For example, if $\lambda=(5,4,4,2)$ we start with the usual Young diagram,
\[\yng(5,4,4,2)\]
and then we rotate to get the following diagram.
\begin{center}\label{gr:contour}
\includegraphics{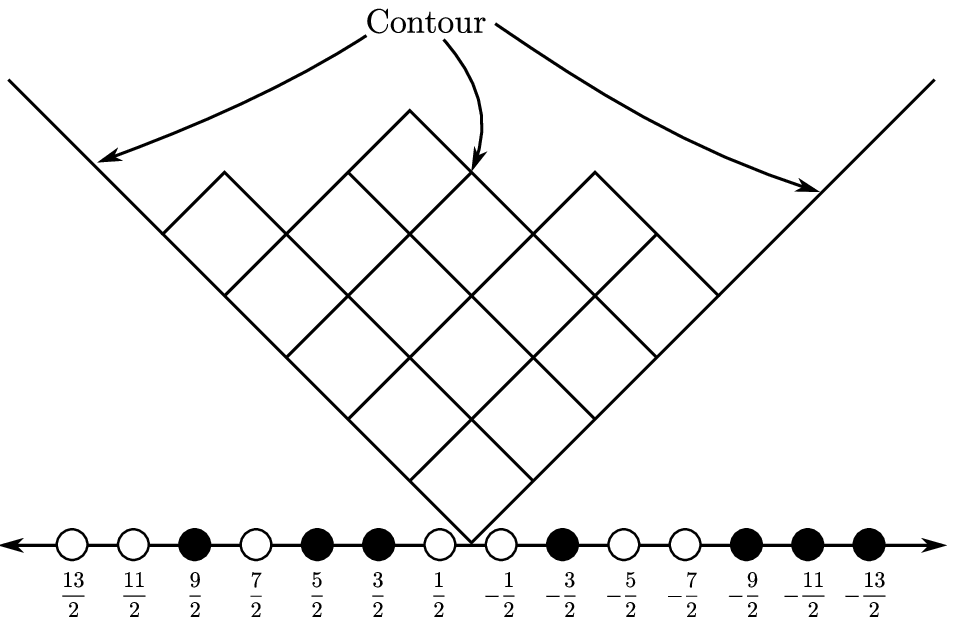}\label{fig:beadsgrid}
\end{center}
In the diagram we have also drawn the rotated axes, which extend to infinity. The \emph{contour} $L_\lambda$ of $\lambda$ is the graph of the continuous function that describes the piecewise-linear curve starting (in the picture) with the diagonal axis on the left, running along the top border (or rim) of the diagram of the partition and ending with the diagonal axis on the right.
It is clear that the contour of $\lambda$ has slope $-1$ wherever there is a white pebble, and has slope $+1$ in the intervals where  there is a black pebble, that is, on the segments corresponding to the shifted coordinates $\xi_i$ of $\lambda$. The association of sequences of pebbles to a Young diagram is known as a \emph{Maya diagram}. Observe that in such a diagram there are always as many black pebbles to the left of zero as there are white pebbles to the right of zero.

If we assign the number 0 to the white pebbles and the number 1 to the black ones, we get a sequence $\{c_i\}_{i\in \Z}$, such that $c_i\in\{0,1\}$, $c_i=0$ for all $i\ll 0$, and $c_i=1$ for all $i\gg 0$. In the case of the example above the sequence equals:
\[\dots,0,0,0,1,0,1,1,0,0,1,0,0,1,1,1,\dots\]
 The sequence completely determines the partition $\lambda$. It coincides with the sequence one gets if one assigns the number 1 to each vertical segment  and the number 0 to each horizontal segment in the contour, as follows:
\begin{center}
\includegraphics{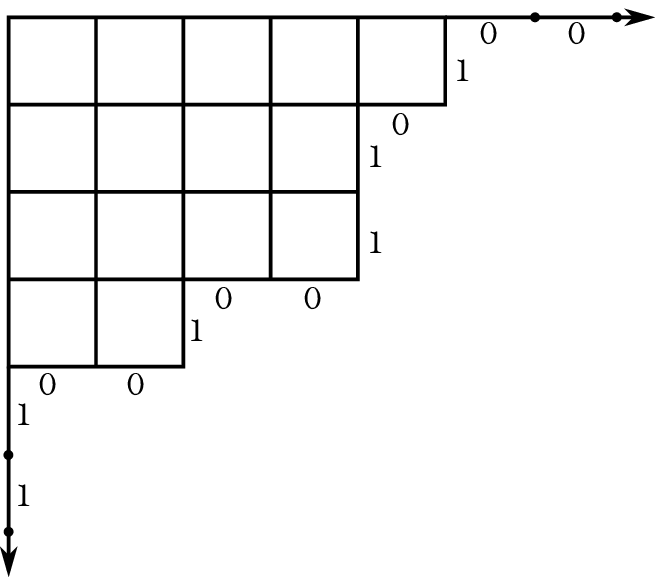}
\end{center}
Note that two sequences $\{a_i\}_{i\in\Z}$ and $\{b_i\}_{i\in\Z}$ of zeros and ones with the above properties determine the same partition if they are translates $a_i=b_{i-n}$ of each other, for some $n\in \Z$. Both points of view ---Maya diagrams and binary sequences--- are equivalent. In what follows, we will prefer the language of pebbles placed under the rotated diagram.

The space $\Lambda_0^{\frac{\infty}2}V$ is contained inside \emph{fermionic Fock space} $\Lambda^{\frac{\infty}2}V$, the space generated by \emph{all} wedge products of the form
\begin{equation}\label{eq:generator}
v_{\{s_i\}}=\underline{ s_1}\wedge \underline{ s_2}\wedge \underline{ s_3}\wedge\cdots,
\end{equation}
where $s_i$ is a half-integer contained in $\Z+\frac12$, $s_1\geq s_2\geq s_3\geq\cdots$, and $s_i-1=s_{i+1}$ for $i$ large enough. The 0-charge sector is simply the subspace whose generators can be identified with partitions. The basis of $\Lambda^{\frac{\infty}2}_0$ given by vectors of the form \eqref{eq:generator} is characterized by containing only those generators for which there are as many positive numbers $s_i$ present as there are negative numbers $s_i$ missing from the wedge.

The fermionic Fock space is a formalization of the quantum-mechanical model known as \emph{Dirac's sea of electrons}. In this model, the ground state corresponds to the vacuum, represented by the vector \[v_\varnothing=\underline{-\tfrac12}\wedge\underline{-\tfrac32}\wedge \underline{-\tfrac52}\wedge\cdots.\]
We say that the state represented by this vector corresponds to the vacuum because it has no electrons (factors $\underline i$ with $i>0$) and no positrons (missing factors $\underline i$ with $i<0$).
Electrons can be created by exchanging a white pebble with a black pebble; a black pebble on a positive number represents an electron excited to the energy given by that number, and a white pebble on a negative number represents a positron\footnote{I am very grateful to the anonymous referee who has clarified some history to me: ``Dirac originally identified the holes as protons. Dirac himself said many times that this error was due to his lack of imagination. This does not work at all. Later the holes were identified as prositrons and the idea of anti-particle was introduced for the first time in physics. [\dots] That works better, but is not considered nowadays to be the best way to think of positrons.'' Apparently the identification was not due to Dirac.}. A vector in $\Lambda^{\frac\infty2}_0$ corresponds to a state with equal number of electrons and positrons, and has zero net charge. The number of cells in the corresponding partition can be defined to be equal to the total energy of the system; see below.

\subsection{Basic operators}
Let $k$ be a number in $\Z+\frac12$ and let $\psi_k$ be the \emph{creation operator} in $\Lambda^{\frac\infty2}V$ defined by
\[\psi_k(v)=\underline k\wedge v.\]
Note that $\underline k$ is added at the very beginning of the sequence, so the anti-com\-mu\-ta\-tiv\-i\-ty may produce a sign (when comparing with the elements of the basis). For example,
\begin{align*}
\psi_{\frac32}(v_{(3,1)})&=\underline{\tfrac32}\wedge \left(\underline{\tfrac52}\wedge\underline{-\tfrac12}\wedge \underline{-\tfrac52}\wedge\underline{-\tfrac72}\wedge\cdots \right)\\
 &=- \left(\underline{\tfrac52}\wedge\underline{\tfrac32}\wedge\underline{-\tfrac12}
 \wedge\underline{-\tfrac52}\wedge\underline{-\tfrac72}\wedge\cdots\right)
\end{align*}
The anti-commutativity also implies that if the factor $\underline k$ is there already, the result is zero. For instance, $\psi_{\frac52}(v_{(3,1)})=0$.
We define an inner product in $\Lambda^{\frac\infty2}V$,
\[\left\langle v_{\{s_i\}},v_{\{t_i\}}\right\rangle=\delta_{\{s_i\},\{t_i\}},\]
where $\delta_{\{s_i\},\{t_i\}}$ is 1 if both sequences $s_i$ and $t_i$ are equal, and 0 otherwise. Then $\psi_k$ has an adjoint operator $\psi_k^*$, known as the \emph{annihilation operator}. The idea is that, if the spot $k$ is empty (i.e., occupied by a white pebble), $\psi_k$ adds an electron (represented by a black pebble) there, and $\psi^*_k$ removes it. These operators satisfy the anti-commutation relations
\[\psi_i\psi_j^*+\psi_j^*\psi_i=\delta_{ij},\]
\[\psi_i\psi_j+\psi_j\psi_i=0=\psi_i^*\psi_j^*+\psi_j^*\psi_i^*.\]

Define the \emph{normally ordered product}, a device invented to make it easy to keep track of the convergence of infinite sums of products of the operators $\psi_i$ and $\psi^*_i$, by
\[:\psi_i\psi_j^*:=\left\{\begin{array}{ll}\psi_i\psi^*_j,&j>0,\\-\psi^*_j\psi_i,&j<0.\end{array}\right.\]
If $i>j$, the net effect of applying $:\psi_i\psi_j^*:$ to a vector $v_\lambda$, where $\lambda$ is some partition, is to add a strip to $\lambda$ if possible, and to add a sign according to the parity of the height of the strip (this should be reminiscent of equation \eqref{murnakruleeq}). For example,
\[:\psi_{\frac{11}2}\psi^*_{-\frac32}:v_{(5,4,4,2)}=-v_{(6,6,5,5)}.\]
This is illustrated in the following picture, which should be compared to the diagram on page \pageref{fig:beadsgrid}:
\begin{center}
\includegraphics{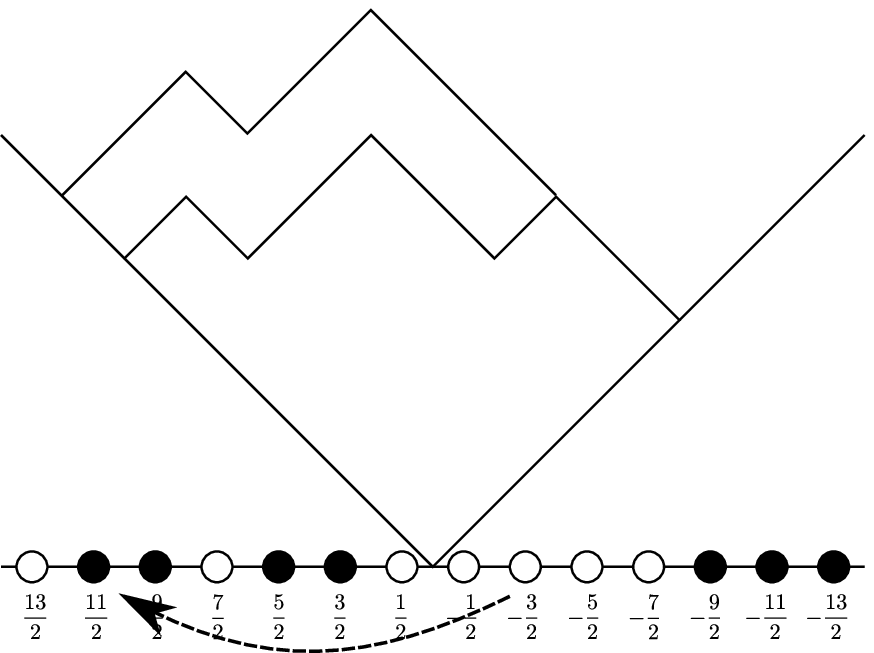}
\end{center}
The height of the strip $(6,6,5,5)/(5,4,4,2)$ added in this procedure is 3, so the sign is indeed $(-1)^3=-1$:
 \[:\psi_{\frac{11}2}\psi^*_{-\frac32}:v_{(5,4,4,2)}=-v_{(6,6,5,5)}.\]
 If the addition of the strip is not possible, the result is zero.
 If $i<j$, the effect is precisely the opposite: a strip is removed and again a sign is added according to the parity of the height of the strip; if the strip cannot be removed, the result is zero.

Let us give some examples to show how useful a device the normally ordered product turns out to be.
Right away, one can define the \emph{energy operator} $H$ by
\[H=\sum_{k\in \Z+\frac12}k\,:\psi_k\psi^*_k:\]
and it is a quick exercise to convince oneself that $Hv_\lambda=|\lambda|v_\lambda$. For any given partition $\lambda$, the sum involved in $Hv_\lambda$ is actually finite (and hence there is no doubt that it converges) because of the way the normally ordered product works. The operator $H$ is often called the \emph{Hamiltonian}, whence the notation $H$.

Another interesting operator is the \emph{charge operator} $C$ given by
\[C=\sum_{k\in Z+\frac12}:\psi_k\psi^*_k:\]
whose action should be considered in the larger space $\Lambda^{\frac\infty2}$:
\begin{multline*}
C(\underline{s_1}\wedge \underline{s_2}\wedge\cdots)=[(\textrm{\# of present positive $s_i$})\\-(\textrm{\# of missing negative $s_i$})]\underline{s_1}\wedge \underline{s_2}\wedge\cdots.
\end{multline*}
In particular, the vectors in the 0-charge sector $\Lambda^{\frac{\infty}2}_0$ are characterized by being annihilated by $C$, as their name suggests.

Let $n>0$ be an integer. We want to define operators $\alpha_{-n}$ and $\alpha_n$ that will, loosely speaking, respectively increase and decrease the energy of an electron by $n$ in all possible ways or, equivalently, they will respectively add and remove all possible strips of length $n$. We let, for $n$ in $\Z$,
\[\alpha_{-n}=\sum_{k\in \Z+\frac12}:\psi_{k+n}\psi_k^*:\]
For example,
\[\alpha_{-3}\,v_{(3,1)}=v_{(6,1)}-v_{(3,2,2)}+v_{(3,1,1,1,1)},\]
which graphically corresponds to:
\[\alpha_{-3}\left(\yng(3,1)\right)=\young(\hfil\hfil\hfil\bullet\bullet\bullet,\hfil) -\young(\hfil\hfil\hfil,\hfil\bullet,\bullet\bullet) +\young(\hfil\hfil\hfil,\hfil,\bullet,\bullet,\bullet).\]
The bullets $\bullet$ mark the cells corresponding to the added strips. Note that the signs correspond to the parities of the heights of these strips. Similarly,
\[\alpha_3\,v_{(5,4,3)}=v_{(5,4)}-v_{(3,3,3)}-v_{(5,2,2)},\]
or
\[\alpha_3\left(\yng(5,4,3)\right)=\yng(5,4)-\yng(3,3,3)-\yng(5,2,2).\]
Strips of the forms $\yng(3)$ and $\yng(2,1)$ have been removed, and again the signs correspond to the parity of their heights.

 Let $\varnothing$ stand for the empty partition. The vector $v_\varnothing=\underline{-\frac12}\wedge\underline{-\frac32}\wedge\underline{-\frac52}\wedge\cdots$ is known as the \emph{ground state}.

 Let $\mu$ be a partition. The following important identity is an immediate consequence of the Murnag\-han-Nakayama rule, equation \eqref{murnakruleeq}:
 \begin{equation}\label{eq:murnakalphas}
 \boxed{\alpha_{-\mu_1}\alpha_{-\mu_2}\cdots\alpha_{-\mu_k}v_\varnothing=\sum_\lambda \chi^\lambda(\mu)\,v_\lambda.}
 \end{equation}
 The sum is over all partitions $\lambda$, but of course the character $\chi^\lambda(\mu)$ equals zero for all but finitely many of them. Since the characters $\chi^\lambda$ are linearly independent from each other, the vectors $\prod_i \alpha_{-\mu_i}v_\varnothing$ form a basis of the entire space $\Lambda^{\frac\infty2}_0V$.

The space $\Lambda^{\frac\infty2}_0V$ is almost\footnote{The topological completion of $\Lambda^{\frac{\infty}2}_0$ with respect to the norm induced by this inner product is indeed a Hilbert space, but $\Lambda^{\frac\infty2}_0$ is not complete.} a Hilbert space when endowed with the inner product $\langle\cdot,\cdot\rangle$ induced by
\[\langle v_\lambda,v_\mu\rangle=\delta_{\lambda\mu},\]
where $\delta_{\lambda\mu}$ is 1 when the partitions $\lambda$ and $\mu$ are equal, and vanishes otherwise.
With this inner product, the adjoint $\alpha^*_n$ of $\alpha_n$ satisfies
\begin{equation*}
\alpha_n^*=\alpha_{-n},
\end{equation*}
and we have the commutation relations
\begin{equation}\label{eq:commutator}
[\alpha_k,\alpha_r]:=\alpha_k\alpha_r-\alpha_r\alpha_k=k\,\delta_{-k,r}.
\end{equation}

\subsection{$p$-cores and $p$-quotients}
\label{sec:pquotients}
The Young diagram of the partition $(3,2,2,2)$ can be constructed out of 3-strips, but the Young diagram of the partitions $(4,2,2,1)$ cannot. The theory of $p$-quotients and $p$-cores gives a framework to understand how this works, and the infinite wedge clarifies it further. We first illustrate with the $p=2$ case, and we comment on the general case at the end of the section. We mainly follow \cite[Excercise 7.59]{stanley}, \cite[Example I.1.8]{macdonald}.

It is easy to identify, given the Maya diagram of a partition, the sites from which it is possible to remove a $p$-strip: one only needs to look for a site with a black pebble such that in the site $p$ units to the right there is a white pebble. In the example above involving partition $(6,6,5,5)$, since we have a black pebble at the site $\frac{11}2$, we can remove strips of sizes 2, 5, 6, 7, 8, and 9, corresponding respectively to the white pebbles at the sites $\frac 72$, $\frac12$, $-\frac12$, $-\frac32$, $-\frac 52$, and $-\frac72$.

Note that there is an immediate correspondence between the strips we can remove and the hooks of the partition. (The hooks and their lengths are defined in Section \ref{sec:hookformula}.) Indeed, if instead of joining the left-most and right-most cells of the strip along the rim, we look at the hook whose two ends correspond to these two cells, it is easy to see that the lengths of the hook and the strip are exactly the same. In the following diagram, we see the same strip as before, and the corresponding hook shaded in grey.
\begin{center}
\includegraphics{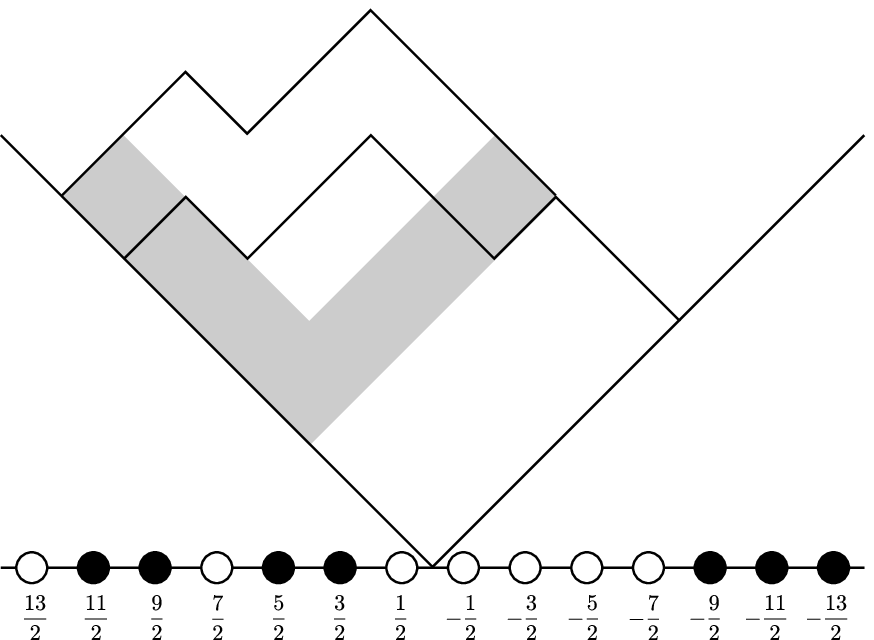}
\end{center}
Both the strip and the hook have length $7=\frac{11}2-\left(-\frac32\right)$.

\subsubsection*{2-quotients.}

Given the sequence of black and white pebbles corresponding to any partition $\lambda$, we can split it into two sequences by picking the sites alternatingly. Namely, we assign to one of these sequences all the pebbles on the sites $2n+\frac 12$, $n\in \Z$, and to the other, the pebbles on the sites $2n-\frac12$, $n\in \Z$. The resulting sequences determine two partitions $\alpha$ and $\beta$ which are together known as the \emph{2-quotient} $(\alpha,\beta)$ of $\lambda$, and they carry information on how $\lambda$ can be built by adjoining a \emph{2-core}, to be defined below, and a bunch of \emph{2-dominoes}, which are simply 2-strips, so they are of the form $\yng(2)$ or $\yng(1,1)$. For example, in the case of the partition $(5,4,4,2)$, we split the pebbles as follows:
\begin{center}
\includegraphics{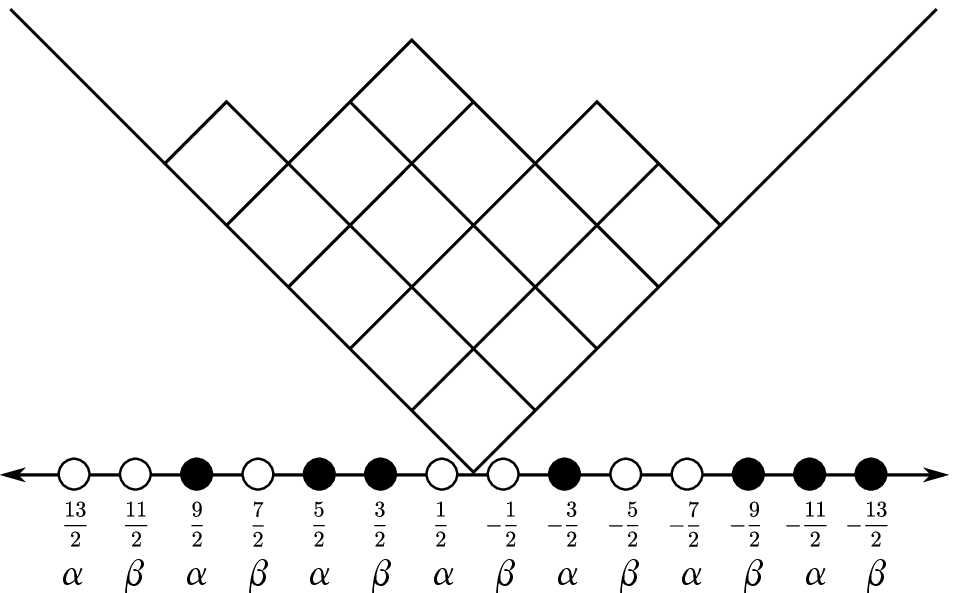}
\end{center}
and we obtain the following two sequences of black and white pebbles:
\begin{center}
\includegraphics{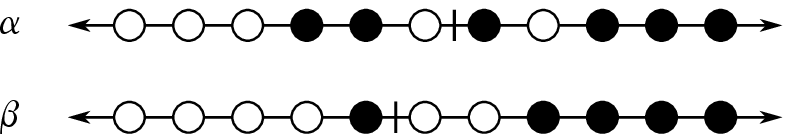}
\end{center}
(Here, the position of zero has been recorded with a vertical line. It is off-center in the sense that the number of black pebbles to the left of it does not equal the number of white pebbles to the right of it. We will discuss this below.)
These in turn correspond to the following partitions:
\[\alpha=\yng(2,2,1)\quad\textrm{and}\quad\beta=\yng(2).\]

Every time we remove a 2-domino (which is the same as a strip of length 2) from $\lambda$, we are removing a cell from one of the components $\alpha$ and $\beta$ of the 2-quotient. This is because we are exchanging two pebbles of opposite color which are exactly 2 units apart, and hence belong in the same component of the 2-quotient.

\subsubsection*{2-cores.}

In the process of removing 2-dominoes, we may get stuck before removing all the cells in the partition $\lambda$. The resulting partition is known as the \emph{2-core} and it is always shaped as a staircase, that is, it is of the form $(n,n-1,n-2,\dots,2,1)$, $n=1,2,\dots$. No 2-dominoes can be removed from a 2-core. The following are the first few 2-cores:
\[\yng(1)\qquad\yng(2,1)\qquad\yng(3,2,1)\qquad\yng(4,3,2,1)\]
It is easy to see, for example, that the 2-core of the partition $(5,4,4,2)$ is $(1)$. For instance, we may remove its 2-dominoes in the order suggested by the following diagram:
\[\young(\hfil6633,7522,7511,44)\]
These moves in turn correspond to the following removals in the 2-quotients:
\[\alpha=\young(63,52,4)\quad \beta=\young(71)\]
In terms of pebble exchanges, the picture is as follows:
\begin{center}
\includegraphics{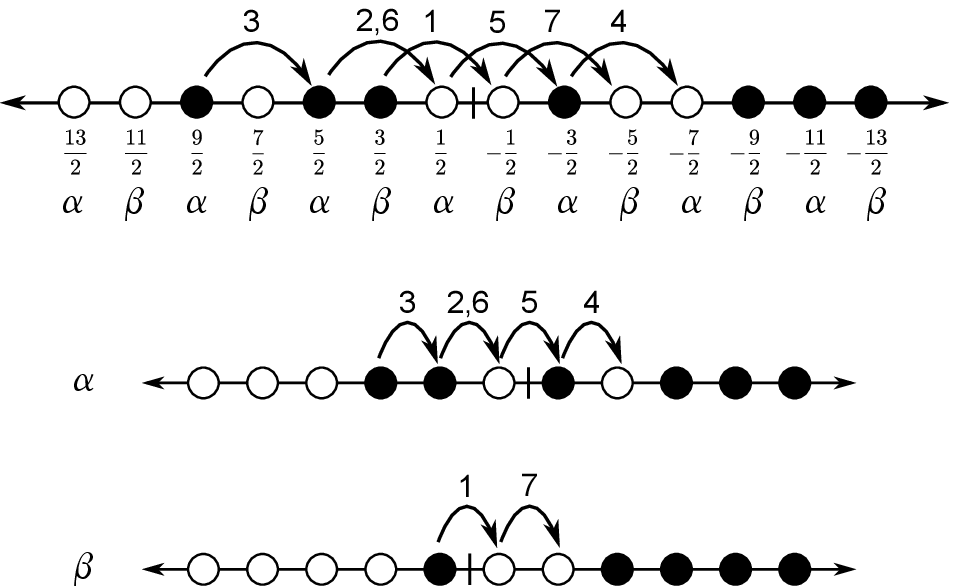}
\end{center}
As we noted before, the vertical bar denoting the original position of zero with respect to the sequences of pebbles for $\alpha$ and $\beta$ is off-centered; this is just a consequence of the presence of the non-trivial 2-core $\yng(1)$. A larger 2-core produces a larger translation of the components of the 2-quotient.

\subsubsection*{The general case.}
In the general case in which $p$ is any positive integer $\geq2$, one uses the same principle: produce $p$ different partitions by taking the chain of pebbles and distributing it into $p$ sets that alternate in the same periodic order throughout $\Z+\frac12$. The $p$-quotient is composed by these partitions.

The part of the partition that remains after removing all possible $p$-strips is the $p$-core. It is reflected in the Maya diagram in the form of translations of the chains of pebbles corresponding to the $p$-quotient away from the center. In other words, the charges of the Maya diagrams of the partitions of the $p$-quotient are what produces the $p$-core. We remark that $p$-cores do not in general look like staircases, as in the $p=2$ case.

An application of this is the following formula for the characters of the representation corresponding to the partition $\lambda$ evaluated at elements with cycle type $(p,p,\dots,p)$: if $(\lambda^1,\lambda^2,\dots,\lambda^p)$ is the $p$-quotient of $\lambda$ and $\mu$ is its $p$-core, then
\[\left|\chi^\lambda(p,p,\dots,p)\right|=\left\{\begin{array}{ll}
0,&\mu\neq\varnothing,\\
\binom{|\lambda|/p}{|\lambda^1|,|\lambda^2|,\dots,|\lambda^p|}
\prod_i \dim\lambda^i,&\mu=\varnothing.
\end{array}\right.
\]
Here,
\[\binom{|\lambda|/p}{|\lambda^1|,|\lambda^2|,\dots,|\lambda^p|}=
\frac{(|\lambda|/p)!}{|\lambda^1|!\,|\lambda^2|!\cdots|\lambda^p|!}.
\]
The formula follows from the Muraghan-Nakayama rule and from the picture outlined above: adding a $p$-strip is equivalent to adding a single cell to one of the components of the $p$-quotient, so the factors $\dim\lambda^i$ appear. The multinomial factor appears because one has to choose which component of the $p$-quotient to affect with each added $p$-strip; in other words, that factor is there because the order matters.

\section{Vertex operators}
\label{sec:traces}
In this section we explore some of the ways known to work with the so-called vertex operators in the half-infinite wedge fermionic Fock space. These operators have found important applications in the field of enumerative combinatorics, and they most likely will continue to do so.

We have also included Section \ref{sec:theta} on theta functions and quasimodular forms because these objects tend to appear naturally in connection with the traces of vertex operators and we will need some of their properties in Section \ref{sec:probability}.

To keep things simple, we will consider only vertex operators of the form
\begin{equation}\label{eq:vertexopdefinition}
\Gamma_{-}\Gamma_{+}
\end{equation}
where
\[\Gamma_{\pm}=\exp\!\left(\sum_{\pm n>0} c_{n}\alpha_{n}\right)\!,\]
$c_n\in \C$,
acting on the space $\Lambda^{\frac\infty2}V$ defined in Section \ref{sec:fockspace}. Of course, this is just notation, and what we really mean is the expansion of the series of the exponential functions. Note that the order in equation \eqref{eq:vertexopdefinition} is not arbitrary: the operators that decrease charge must act first in order to avoid getting diverging sums. 

These operators are not necessarily bounded, but they are well-defined for all vectors in $\Lambda^{\frac\infty2}V$ and, for each such vector, they give a vector in a completion of $\Lambda^{\frac\infty2}V$ where we allow sums of infinitely many generators. In our applications the topology of this completion will not be important, so for simplicity let us be a little sloppy and ignore it.

In a sense, what a vertex operator does is to simultaneously consider all the possible products of characters of the symmetric group (compare equation \eqref{eq:murnakalphas}) with special weights given by the numbers $c_n$.
These vertex operators turn out to be quite common: they are appear naturally when one is dealing with an operator with some prescribed commutation relations, as the following lemma shows.

\begin{lem}[\protect{\cite[Lemma 14.5]{kac}}]
\label{lem:vertexcharacterization}
Let $D$ be a linear operator in $\Lambda^{\frac\infty2}$. Then:
\begin{enumerate}[a.]
\item if $[\alpha_{-n},D]=c_{-n}D$ for all $n=1,2,\dots$, then $D=c\Gamma_+$ for some $c\in \C$,
\item if $[\alpha_{n},D]=c_{n}D$ for all $n=1,2,\dots$, then $D=c\Gamma_-$ for some $c\in \C$, and
\item if $[\alpha_{n},D]=c_{n}D$ for all $n\in\Z\setminus\{0\}$, then $D=c\Gamma_-\Gamma_+$ for some $c\in \C$.
\end{enumerate}
\end{lem}

An important special case of a vertex operator is given by the following useful consequence of the so-called \emph{boson-fermion correspondence}. The motivation for the fancy name is rather lengthy and has already been explained very nicely for example in \cite[Chapter 5]{miwajimbodate}; the idea is essentially that the left-hand-side of equation \eqref{eq:bosongf} below is the generating function for a sequence of operators that have the commutation properties of bosons, while the operators $\alpha_n$ have the commutation properties of fermions. Examples of applications of formula \eqref{eq:bosongf} can be found in Sections \ref{sec:uniform} and \ref{sec:pillowcase}. Let
\[\psi(x)=\sum_{k\in\Z+\frac12}x^k\psi_k\quad\textrm{and}
\quad
\psi^*(x)=\sum_{k\in\Z+\frac12}x^{-k}\psi^*_k,
\]
where $\psi_k, \psi^*_k$ are the creation and annihilation operators defined in Section \ref{sec:fockspace}.
Then:
\begin{multline}\label{eq:bosongf}
\psi(xy)\psi^*(y)=\frac{1}{x^{1/2}-x^{-1/2}}\times\\
\exp\left(\sum_{n>0}\frac{(xy)^n-y^n}n\,\alpha_{-n}\right)
\exp\left(\sum_{n>0}\frac{y^{-n}-(xy)^{-n}}n\,\alpha_n\right).
\end{multline}
This identity can be checked easily using Lemma \ref{lem:vertexcharacterization}, for it is not a hard exercise to check the appropritate commutation relations of $D=\psi(xy)\psi^*(y)$ with the operators $\alpha_n$, and it can be checked that the coefficient $c$ turns out to be 
\[c=\frac1{x^{1/2}-x^{-1/2}}=\sum_{k=0}^\infty x^{-k-1/2}\]
by computing the inner product 
\begin{align*}
c&=\langle v_\varnothing,Dv_\varnothing\rangle=\langle v_\varnothing,\psi(xy)\textstyle\sum_{k\in \Z+\tfrac12}y^{-k}\psi_k^*v_\varnothing\rangle\\
&=\sum_{k=-\frac12,-\frac32,\dots}\langle v_\varnothing,\psi(xy)y^{-k}\psi_k^*v_\varnothing\rangle \\
&=\sum_{\ell\in\Z+\frac12}\sum_{k=-\frac12,-\frac32,\dots}\langle v_\varnothing,(xy)^\ell y^{-k}\psi_\ell\psi_k^*v_\varnothing\rangle\\
&=\sum_{k=-\frac12,-\frac32,\dots}\langle v_\varnothing,(xy)^k y^{-k}\psi_k\psi_k^*v_\varnothing\rangle=\sum_{k=-\frac12,-\frac32,\dots}x^k.
\end{align*}
 For the full proof of equation \eqref{eq:bosongf}, the reader can also see \cite[Theorem 14.10]{kac}.

\subsection{The value of a vertex operator}
\label{sec:vertexopvalue}
The matrix entries of $\Gamma_{\pm}$ can be found as follows \cite[Section 3.1]{pillow}, \cite[Section 2.2.4]{okounkovreshetikhin}.

Note that, from the definition of $\alpha_n$, we see that it can also be applied directly to the vectors $\underline k$ in the basis of $V$, $k\in\Z+\frac12$. Their action is
\[\alpha_n\underline{k}=\underline{k-n},\quad n\in\Z\]
and it follows that $\Gamma_+$ and $\Gamma_-$ also have an action on $\underline{k}$.

Fix a partition $\lambda$. Choose $N$ large enough to have $\lambda_{2N+1}=0$.
Then, when acting on the vectors $\underline{\lambda_i-i+\frac12}$, the operators $\Gamma_+$ and $\Gamma_-$ are lower and upper triangular, respectively, and their diagonal elements are all equal to 1. For this reason, the vectors $\underline{\lambda_i-i+\frac12}$ for $i>2N$ end up being irrelevant for the evaluation of $\langle\Gamma_-\Gamma_+v_\lambda,v_\mu\rangle$. On account of this, it is enough to study what $\Gamma_-\Gamma_+$ does in the exterior power of the finite-dimensional vector space generated by
\[e_k=\underline{-2N+k+\tfrac12},\quad k=0,1,\dots,\lambda_1+2N.\]
For convenience, we will use below the same notation $e_k$ for any $k\in\Z$, with a similar definition.

Consider first the action of the operator $\Gamma_+$, which is encoded in the following generating function
\begin{equation}\label{eq:generatingseriesgammaplus}
\sum_{k,\ell\geq0}x^ky^\ell\left(\Gamma_+e_k,e_\ell\right) =
\sum_{k,\ell}x^ky^\ell\left(\left[\prod_{n=1}^\infty \exp(c_n\alpha_n)\right]e_k,e_\ell\right)
\end{equation}
Since the action of each of these exponentials goes like
\[\exp(c_n\alpha_n)e_k=\sum_{j=0}^\infty \frac{c_n^j}{j!} \alpha_n^je_k=\sum_{j=1}^\infty \frac{c_n^j}{j!}e_{k-nj}\]
we see that, letting $m_i(\nu)$ be the number of parts of length $i$ in the partition $\nu$,
\[\left[\prod_n\exp(c_n\alpha_n)\right]e_k
=\sum_\nu \left(\prod_{i=1}^\infty \frac{c_i^{m_i(\nu)}}{m_i(\nu)!}\right)
e_{k-|\nu|},
\]
where the sum is taken over all partitions $\nu$. But notice also that
\[\sum_{|\nu|=n}\prod_{i=1}^\infty \frac{c_i^{m_i(\nu)}}{m_i(\nu)!}\]
is the coefficient of $z^n$ in the series expansion of
\[\phi^+(z):=\exp\sum_{k=1}^\infty c_kz^k.\]
So we can write the generating function \eqref{eq:generatingseriesgammaplus} in the form
\begin{align*}
\sum_{k,\ell}x^ky^\ell\left[z^{k-\ell}\right]\phi^+(z)=\sum_\ell x^\ell y^\ell\sum_k x^k\left[z^k\right]\phi^+(z)=\sum_\ell x^\ell y^\ell\phi^+(x)=\frac{\phi^+(x)}{1-xy}.
\end{align*}
Here, the notation $[z^m]$ stands for the operation of taking the coefficient of $z^m$ in the Taylor series expansion of the function to the right of this symbol.
Now let $\phi^+_n$ be the coefficients of the series expansion of $\phi^+(z)$, \[\phi^+(z)=\sum_{n=0}^\infty \phi^+_nz^n.\]
In particular, $\phi^+_0=1$.
Then
\[\left(\Gamma_+e_k,e_\ell\right) = [x^k y^\ell]\frac{\phi^+(x)}{1-xy} = \phi^+_{k-\ell} \]

We plug this into the full half-infinite wedge to conclude that, letting
\[s_{\lambda/\mu}^+=\det\left(\phi^+_{\lambda_i-\mu_j+j-i}\right),\]
we have
\[\langle \Gamma_+v_\mu,v_\lambda\rangle=s^+_{\lambda/\mu}.\]
Here, $s_{\lambda/\mu}^+$ is the skew Schur function $s_{\lambda/\mu}$ (defined in Section \ref{sec:normalschurskewschur}) specialized to the case in which the homogeneous symmetric functions $h_n$ equal $\phi_n^+$ for $n\geq0$ or, equivalently, specialized to the case in which the symmetric power functions $p_n$ equal $nc_n$ (compare with the relations \eqref{eq:skewschurh}). In particular, $s_{\lambda/\mu}=0$ if the Young diagram of $\mu$ is not completely contained in the Young diagram of $\lambda$.

A similar calculation shows that
\[\langle\Gamma_-v_\mu,v_\lambda\rangle=s^-_{\mu/\lambda}\]
where
\[s^-_{\mu/\lambda}=\det\left(\phi^-_{\mu_i-\lambda_j+j-i}\right),\]
and $\phi^-_n$ are the coefficients in the series expansion of \[\phi^-(z)=\exp\sum_{n=1}^\infty c_{-n}\alpha_{-n}.\]

With this information in hand, it is immediate that
\[\boxed{\left\langle\Gamma_-\Gamma_+v_\mu,v_\lambda\right\rangle= \sum_{\nu} s_{\lambda/\nu}^- s_{\mu/\nu}^+.}\]

As a side note, whenever one can write $\phi(z)=\phi^+(z)\phi^-(z)$, where $\phi^+$ and $\phi^-$ are analytic and non-vanishing in some neighborhood of, respectively, the interior and the exterior of the unit disc, this expression is called the \emph{Wiener-Hopf factorization} of the function $\phi$.
Clearly, this type of factorization may be important when we want to understand vertex operators. (However, the analyticity of $\phi^+$ and $\phi^-$ need not be true in general.) See \cite{okounkovreshetikhin} for an application to the enumeration of random surfaces.

\subsection{Formulas for the computation of the trace of a vertex operator}
\label{sec:traceformulas}
Sometimes one needs to compute the trace of a vertex operator $T$. In such cases, there are two formulas that are quite useful.

The first one is the obvious decomposition of the space
\[\Lambda^{\frac\infty2}_0V\cong\bigotimes_{n=1}^\infty\bigoplus_{k=0}^\infty
\alpha_{-n}^kv_\varnothing,\]
whose immediate consequence is that
\begin{equation}\label{eq:tracetoprod}
\boxed{\trace T=\prod_{n=1}^\infty\trace \left(\left.T\right|_{\bigoplus_{k=0}^\infty\alpha_{-n}^k v_\varnothing}\right).}
\end{equation}
This, together with the commutation relation \eqref{eq:commutator}, allows for the following formula \cite{pillow} to be used: for complex numbers $A$ and $B$,
\begin{equation}\label{eq:specialtrace}
\boxed{\left.\trace q^He^{A\alpha_{-n}}e^{B\alpha_n}\right|_{\bigoplus_{k=0}^\infty\alpha_{-n}^k v_\varnothing}=\frac{1}{1-q^n}\exp\frac{nABq^n}{1-q^n}.}
\end{equation}
Here, $H(v_\lambda)=|\lambda|v_\lambda$ as above, $q$ is a parameter, and the operator $q^H$ is defined by
\[q^Hv_\lambda=|\lambda|\,v_\lambda,\]
and extended by linearity.
The introduction of $q^H$ is often a good idea because it sometimes simplifies formulas and it produces generating functions that are indexed on the size of the partitions.

To prove \eqref{eq:specialtrace}, one needs to expand both sides. Let us explain this in detail. We have
\begin{align*}
\trace q^He^{A\alpha_{-n}}&\left.e^{B\alpha_n}
\right|_{\bigoplus_{k=0}^\infty\alpha_{-n}^k v_\varnothing}=\sum_{k=0}^\infty\left\langle q^He^{A\alpha_{-n}}e^{B\alpha_n} \alpha_{-n}^k v_\varnothing,\alpha^k_{-n}v_\varnothing\right\rangle\\
&=\sum_k\left\langle q^H\sum_\ell\frac{(A\alpha_{-n})^\ell} {\ell!}\sum_m\frac{(B\alpha_n)^m}{m!}
\alpha_{-n}^kv_\varnothing,\alpha_{-n}^kv_\varnothing\right\rangle\\
&=\sum_{k,\ell,m}q^{n(\ell-m+k)}\frac{A^\ell B^m}{\ell!\,m!}
\left\langle\alpha^\ell_{-n}\alpha^m_{n}\alpha^k_{-n}v_\varnothing, v\alpha^k_{-n}v_\varnothing\right\rangle
\end{align*}
Note that this inner product vanishes when the partitions associated to the vectors involved do not have the same number of cells.  This means that only the terms in which $m=\ell$ survive and we get
\[\sum_{k,\ell}q^{kn}\frac{(AB)^\ell}{\ell!^2}\left\langle\alpha^\ell_{-n}\alpha^\ell_{n}
\alpha^k_{-n}v_\varnothing,\alpha^k_{-n}v_\varnothing\right\rangle.\]
One then plays a game of moving the fermionic annihilators $\alpha_n$ to the right, using the commutation relation \eqref{eq:commutator} and the fact that $\alpha_nv_\varnothing=0$. For example,
\begin{multline*}
\alpha_n\alpha_{-n}\alpha_{-n} v_\varnothing=\alpha_{-n}\alpha_{n}\alpha_{-n} v_\varnothing+n\alpha_{-n}v_\varnothing=\\
\alpha_{-n}\alpha_{-n}\alpha_{n} v_\varnothing+2n\alpha_{-n}v_\varnothing=2n\alpha_{-n}v_\varnothing.
\end{multline*}
In general, one gets
\[\alpha_n^a\alpha_{-n}^bv_\varnothing=\left\{\begin{array}{ll}
\tfrac{b!}{(b-a)!}n^a\alpha_{-n}^{b-a}v_\varnothing&\textrm{if $a\leq b$},\\
0& \textrm{otherwise.}
\end{array}\right.\]
Therefore, the above sum becomes
\[\sum_{k=0}^\infty\sum_{\ell=0}^k q^{kn}\frac{(nAB)^\ell}{\ell!}\binom{k}{\ell}.\]
This equals the right hand side of \eqref{eq:specialtrace}; just invert the order of summation and notice that
\[\sum_{k=\ell}^\infty q^{kn}\binom{k}{\ell}=\frac{q^{n\ell}}{(1-q^n)^{\ell+1} }.\]

\subsection{Theta functions and quasimodular forms}
\label{sec:theta}

It is an experimental fact that, due to the form of formulas like \eqref{eq:bosongf} or \eqref{eq:specialtrace}, traces of vertex operators have a tendency to be expressible in terms of the Jacobi theta function. For examples, please see Section \ref{sec:probability} or, in a more general setting, \cite{kac}.

\subsubsection{The Jacobi theta function}

The \emph{Jacobi theta function}
\begin{equation}\label{eq:jacobithetadef}
\vartheta(x,q)=(q^{1/2}-q^{-1/2})\prod_{i=1}^\infty\frac{(1-q^ix)(1-q^i/x)}{(1-q^i)^2},
\end{equation}
is defined for arbitrary $x$ in $\C$ and $q$ in the unit disc $\{|q|<1\}\subset\C$. In this section we recall some of its properties.

The first property is its modularity. Let $\tau$ and $z$ be defined (modulo $\Z\subset\C$) by $q=e^{2\pi i \tau}$ and $x=e^{2\pi i z}$. It is traditional to abuse notation and write $\vartheta(z,\tau)$ for $\vartheta(x,q)$, and we will embrace this tradition here. As a consequence of the Poisson summation formula, $\vartheta$ satisfies
\[\vartheta\!\left(\frac z\tau,-\frac1\tau\right)=-\frac{e^{\pi i z^2/\tau}}{i\tau}\,\vartheta(z,\tau).\]

Another property is its additive expression, which can be derived using the \emph{Jacobi triple product}, which is the identity
\[\prod_{m=1}^\infty \left(1-x^{2m}\right)\left(1+x^{2m-1}y^2\right)\left(1+x^{2m-1}y^{-2}\right)=\sum_{n\in \Z}x^{n^2}y^{2n},\]
valid for all complex numbers $x$ and $y$ with $|x|<1$ and $y\neq0$. It follows that
\[\vartheta(x,q)=\eta^{-3}(q)\sum_{n\in \Z}(-1)^nq^{\frac{\left(n+\frac12\right)^2}{2}}x^{n+\frac12},\]
where $\eta$ stands for the \emph{Dedekind eta function},
\[\eta(q)=q^{1/24}\prod_{n=1}^\infty(1-q^n).\]

A second property we want to record for later reference is the expression of our $\vartheta$ in terms of the more classical Jacobi thetas, namely, \cite[page 17]{mumford}
\begin{align*}
\vartheta_{00}(z,\tau)&=\sum_{n\in\Z}\exp(\pi i n^2\tau+2\pi i n z) \\
\vartheta_{01}(z,\tau)&=\sum_{n\in\Z}\exp\!\left(\pi i n^2\tau+2\pi i n\left(z+\tfrac12\right)\right)\\
 &=\vartheta_{00}\!\left(z+\tfrac12,\tau\right)\\
 \vartheta_{10}(z,\tau)&=\sum_{n\in\Z}\exp\left(\pi i \left(n+\tfrac12\right)^2\tau+2\pi i \left(n+\tfrac12\right)z\right)\\
 &=\exp\!\left(\frac{\pi i \tau}4+\pi iz\right)\vartheta_{00}\!\left(z+\tfrac12 \tau,\tau\right)\\
 \vartheta_{11}(z,\tau)&=\sum_{n\in\Z}\exp\!\left(\pi i\left(n+\tfrac12\right)^2\tau+2\pi i\left(n+\tfrac12\right)\left(z+\tfrac12\right)\right) \\
 &=\exp\!\left(\frac{\pi i \tau}4+\pi i\left(z+\tfrac12\right)\right)\vartheta_{00}\!\left(z+\tfrac12(1+\tau),\tau\right)
\end{align*}
It follows that
\begin{equation}\label{eq:classictheta}
\vartheta(z,\tau)=-i\eta^{-3}(\tau)\,\vartheta_{11}(z,\tau).
\end{equation}
We further record the modular behavior of the traditional Jacobi thetas \cite[page 36]{mumford}:
\begin{align*}
\vartheta_{00}\!\left(\frac z\tau,-\frac1\tau\right)&=(-i\tau)^{\frac12}\exp\!\left(\frac{\pi i z^2}\tau\right)\vartheta_{00}(z,\tau), \\
\vartheta_{01}\!\left(\frac z\tau,-\frac1\tau\right)&=(-i\tau)^{\frac12}\exp\!\left(\frac{\pi i z^2}\tau\right)\vartheta_{10}(z,\tau), \\
\vartheta_{10}\!\left(\frac z\tau,-\frac1\tau\right)&=(-i\tau)^{\frac12}\exp\!\left(\frac{\pi i z^2}\tau\right)\vartheta_{01}(z,\tau), \\
\vartheta_{11}\!\left(\frac z\tau,-\frac1\tau\right)&=-(-i\tau)^{\frac12}\exp\!\left(\frac{\pi i z^2}\tau\right)\vartheta_{11}(z,\tau).
\end{align*}
Note the additional minus sign in the rule for $\vartheta_{11}$. These identities are useful because it is likely to find quotients of the form $\vartheta(x)/\vartheta(-x)$ involved in partition functions.

The statements of the following two lemmas are more readable if we use $h$ instead of $\tau$, where
\[q=e^{-h}=e^{2\pi i\tau}.\]

\begin{lem}[\protect{\cite[Proposition 4.1]{branchedcoverings}}]\label{lem:theta1}
We have
\[\frac{\vartheta(e^{u},e^{-h})}{\vartheta'(0,e^{-h})}\approx
h\frac{\sin(\pi u/h)}{\pi}\exp\left(\frac{u^2}{2h}\right)\]
as $h\to+0$, uniformly in $u$. This asymptotic relation can be differentiated any number of times.
\end{lem}
\begin{proof}
We have
\[\vartheta(e^{hu},e^{-h})=i\sqrt{\frac{2\pi}{h}}\exp\left(\frac{hu^2}{2}\right)
\vartheta\left(e^{-2\pi i u},e^{-\frac{4\pi^2}{h}}\right),\]
so, expanding the series of $\vartheta$,
\[\vartheta\left(e^{-2\pi i u},e^{-\frac{4\pi^2}{h}}\right)=\sum_{n\in\Z}(-1)^n\exp
\left(-\frac{2\pi^2\left(n+\frac12\right)^2}h\right)e^{-2\pi\left(n+\frac12\right)u}.\]
From here it is clear that, as $h\to+0$, the terms $n=0$ and $n=-1$ dominate all others.
\end{proof}

\begin{lem}\label{lem:theta2}
We have the following approximations for $\vartheta$ at $x=1$ and $x=-1$, respectively:
\[
\begin{array}{c}
\vartheta(e^u,e^{-h})
\approx i\eta^{-3}(q)\sqrt{\frac{2\pi}h}\exp\left(\pi i \left(\left[\frac{u}{2\pi i}-\tfrac12\right]+\tfrac12\right)-\frac{2\pi^2}h\left\{\frac{u}{2\pi i}-\tfrac12\right\}^2\right),
\\
\vartheta(-e^u,e^{-h})
\approx i\eta^{-3}(q)\sqrt{\frac{2\pi}h}\exp\left(-\frac{2\pi^2}h\left\{\frac{u}{2\pi i}-\tfrac12\right\}^2\right),
\end{array}
\]
as $h\to0$, where $q=e^{-h}$, $[x]$ stands for the integer closest to $x$,  and $\{x\}=x-[x]$.
\end{lem}
\begin{proof}[Proof of Lemma \ref{lem:theta2}]
This is again a straightforward application of the modular transformation $h\mapsto -1/h$ in the expression \eqref{eq:classictheta}. We also use the identity \[\vartheta_{11}(-e^{u},q)=-\vartheta_{10}(e^u,q).\qedhere\]
\end{proof}

\subsubsection{Quasimodular forms}
\label{sec:quasimodular}

We follow \cite[Section 3.3]{pillow}.

A \emph{quasimodular form} for a subgroup $\Gamma\subset \SL_2(\Z)$ is the holomorphic part of an almost holomorphic modular form for $\Gamma$. A function of $|q|<1$ is called \emph{almost holomorphic} if it is a polynomial in $(\log|q|)^{-1}$ with coefficients that are holomorphic functions of $q$.

Quasimodular forms constitute a graded algebra $\mathcal{QM}(\Gamma)$. It is a theorem of M. Kaneko and D. Zagier \cite{kanekozagier} that the space of quasimodular forms is simply
\[\mathcal{QM}(\Gamma)=\Q[E_2]\otimes\mathcal M(\Gamma),\]
where $\mathcal M(\Gamma)$ denotes the space of modular forms with respect to the subgroup $\Gamma$ and $E_2$ denotes the first of the Eisenstein series, which are in general defined by
\begin{equation*}
E_{2k}(q)=\frac{\zeta(1-2k)}2+\sum_{n=1}^\infty\left(\sum_{d|n}d^{2k-1}\right)q^n,\quad k=1,2,\dots
\end{equation*}
In particular,
\begin{equation}\label{eq:certaineisensteinseries}
E_2(q),E_2(q^2),E_4(q^2)\in\mathcal {QM}(\Gamma_0(4))
\end{equation}
where
\[\Gamma_0(4)=\left\{\bigl(\begin{smallmatrix}a&b\\c&d\end{smallmatrix}\bigr):c\equiv0\!\!\mod 4\right\}.\]
It is well known that the three series in \eqref{eq:certaineisensteinseries} algebraically generate all other Eisenstein series and their products, so these too are quasimodular forms. Hence, in order to show that a certain function is a quasimodular form, it  is sufficient to show that it is a polynomial in the $E_{2k}$.

Note the close connection between the Eisenstein series and the theta function, best evidenced by the formal identities (easily checked by expanding both sides)
\begin{align*}
\log\frac z{\vartheta(e^z)}&=2\sum_{k\geq1}\frac{z^{2k}}{(2k)!}E_{2k}(q), \\
\log\frac{\vartheta(-e^{z})}{\vartheta(-1)}&=2\sum_{k\geq1}\frac{z^{2k}}{(2k)!}
\left[E_{2k}(q)-2^{2k}E_{2k}(q^2)\right],\\
\vartheta(-1)&=2i\left(\prod_{n\geq1}\frac{1+q^n}{1-q^n}\right)^2
=\frac{\eta(q^2)^2}{\eta(q)^4},
\end{align*}
where $\eta$ is the Dedekind eta function.

\section{Variational methods}
\label{sec:variational}
\subsection{The entropy}
\label{sec:entropy}
When thinking about very large partitions, as we will see later, one often finds that most of the partitions in a particular distribution end up looking very much alike. In this case, they are said to have a \emph{limit shape}.

The method to deduce what this limit shape is (if it exists) must vary according to the probability distribution being studied. Probably the easiest case comes from the case of the uniform distribution. For each number $N$, the uniform distribution gives equal probability weight to each of the partitions of $N$. For large $N$, most partitions look like a certain curve, which we shall now deduce.

After rotating the Young diagrams by  $135^\circ$ as in Section \ref{sec:fockspace}, each partition contour becomes a zig-zag curve, which can be interpreted as the graph of a discrete random walk that moves up or down, with slope $+1$ or $-1$, from left to right. For any two points in
\[\{(x,y)\in\Z\times \Z:y\geq|x|\},\]
not both on the same vertical line, there is at least one partition contour that joins them. If the two points are $m>0$ units apart horizontally and $n\geq0$ units apart vertically, there are, in fact
\[\binom{m}{(m+n)/2}\]
possible random walks joining them, each of them determined by the choice of $(m+n)/2$ steps upwards, for the remaining $(m-n)/2$ steps will be taken downwards. (The plot of the random walk is part of a partition's contour.) As we rescale maintaining the overall slope $p=(m+n)/2m$ constant, we find that an application of Stirling's formula gives
\[\frac1m\log\binom{m}{pm}\to S(p)\]
where
\[S(p)=-p\log p-(1-p)\log(1-p)\]
is the \emph{Shannon entropy}.

Let us assume that there exists a limit shape, given by the function $f$.
The parameter $p$, which can be interpreted as the probability of the random walk of going up, then corresponds to $(1+f'(x))/2$. For instance, if $f'(x)=1$, then at the microscopic level the random walk must be going always upwards, so $p=1$. And if $f'(x)=-1$, the random walk is never going up, so it is always going down and $p=0$.

From our calculation above it follows that a heuristic count of partitions (or random walks) around the graph of $f$ is given by
\begin{equation}\label{eq:entropyintegral}
\frac1\delta\int_0^T S\left(p\right)dx=
\frac1\delta\int_0^T S\left(\frac{1+f'(x)}2\right)dx,
\end{equation}
where $\delta$ is the size of the side of the rescaled cells of the Young diagrams. The limit shape would be the curve $f$ maximizing this count, because it would need to have as many of these zigzagging random walks near to it as possible (it must be the look-alike of most large partitions). If one variationally looks for a function $f$ that will maximize this integral, while respecting the restriction that $\int f(x)-|x|\,dx=1$, one ends up finding the correct limit shape for the uniform distribution; see Section \ref{sec:uniform}.

This is really a textbook case, and it can be studied to greater depth. In particular, it is not difficult to find a full large deviation principle. See for example \cite{okounkovlimitshapes}.

\subsection{Approximating products of hooks}
\label{sec:hooksvariational}
In order to deduce the asymptotic behavior of probability distributions in the space of Young diagrams given by products or quotients of hooks, people sometimes use variational arguments resembling the idea that S. Kerov and A. Vershik \cite{vershikkerov1985} and B. Logan and L. Shepp \cite{loganshepp} used in the analysis of the Plancherel distribution. Here, we review this technique.

The idea is that we can approximate the logarithm of products of hook lengths of very large partitions by an integral.

For a non-increasing function $F:\R_+\to\R_+$, define
\[F^{-1}(y)=\inf\{x:F(x)\leq y\},\]
and
\[h_F(x,y)=F(x)+F^{-1}(y)-x-y,\]
which gives an approximation of the hook length at $(x,y)$.
Let $\lambda$ be a partition, and $n=|\lambda|$. We contract its diagram until it has area 1, rescaling by $1/\sqrt {n}$, and we associate to it a function $F(x)$ that describes its rim,
\[F(x)=\frac 1{\sqrt n}\#\{\textrm{parts of $\lambda$ of size $\leq\lceil \sqrt n\,x \rceil$}\}.\]
In the following diagram, we have the example of $\lambda=(4,3,3,2)$.
\begin{center}
\includegraphics{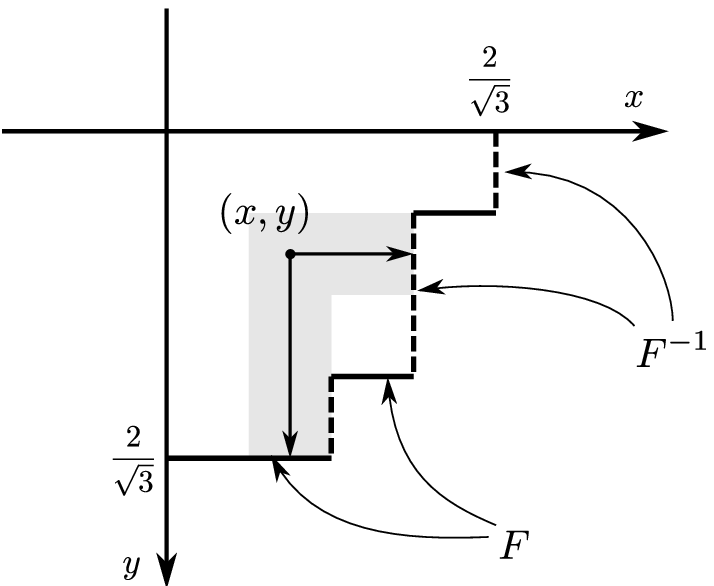}
\end{center}
In the picture, we have inverted the orientation of the $y$ axis to reflect the usual orientation of our Young diagrams, and we have drawn the parts of the rim of $\lambda$ that are described by $F$ with a solid line, while the parts that correspond do $F^{-1}$ are drawn with dashed lines. The numbers on the axes reflect the rescaling by $1/\sqrt{n}=1/\sqrt{12}$. We have also highlighted the hook corresponding to the cell $(2,2)$. We will approximate the logarithm of its length by
\begin{equation}\label{eq:firsthookintegral}
n\int_R \log \left(\sqrt{n} \,h_F(x,y)\right)\,dx\,dy,
\end{equation}
where $R$ is the rectangle between the points $\frac1{\sqrt{12}}(1,1)$ and $\frac1{\sqrt{12}}(2,2)$. To motivate this approximation, we have drawn, at the center of the cell $(2,2)$, the point $(x,y)$ and an inverted L ending extending to the rescaled rim of $\lambda$; the length of this L is precisely $h_F(x,y)$: its vertical part measures $F(x)-y$, while its horizontal part measures $F^{-1}(y)-x$. Then $\sqrt n\,h_F(x,y)=4$ coincides with the length of the hook of cell $(2,2)$, and the integral \eqref{eq:firsthookintegral} is very close to this number as well. In general, we have

\begin{lem}\label{lem:hookapproximation}
Let $\lambda$ be a partition and let $\square\in\lambda$ be a cell in its Young diagram, and denote by $F$ the non-increasing function that describes the rim of the Young diagram of $\lambda$, as defined above. Let $R_\square\subseteq\{(x,y):0\leq y\leq F(x)\}$ be the rectangular domain corresponding to $\square$. Let $h_\square$ denote the hook length of $\square$. Then
\[\log h_\square=n\int_{R_\square}\log\left(\sqrt n h_F(x,y)\right)dx\,dy + c\!\left(h_\square\right),\]
where
\[c(x)=\frac12\sum_{k=1}^\infty\frac{1}{k(k+1)(2k+1)x^{2k}}.\]
\end{lem}
\begin{rmk}\label{rmk:approximation}
The function $c$ is strictly decreasing in the interval $[1,+\infty)$. It remains between $c(1)=\frac12(3-4\log 2)\approx 0.113706$ and
\[\lim_{x\to+\infty} c(x)=0.\]
This implies that the approximation of the logarithm of a product of any $m$ hook lengths using the integral of $\log\left(\sqrt nh_F(x,y)\right)$ will have an error of order $\leq c(1)\cdot m$. This estimate can be improved in the large-scale case; see Lemma \ref{lem:asymptoticsofremainder}.
\end{rmk}
\begin{proof}
Let $(x_0,y_0)$ be the point at the center of $R_\square$. Recall that the area of $R_\square$ is $1/n$. Then
\begin{align*}
n\int_{R_\square} &\log \left(\sqrt n h_F(x,y)\right)\,dx\,dy =
n\int_{R_\square}\log \left(h_\square-\sqrt n(x-x_0)-\sqrt n(y-y_0)\right)dx\,dy \\
&=\log h_\square+n\int_{R_\square} \log\left(1-\frac{(x-x_0)+(y-y_0)}{h_\square/\sqrt n}\right)dx\,dy \\
&=\log h_\square +  \int_{\sqrt n y_0-\frac12}^{\sqrt n y_0+\frac12}\int_{\sqrt n x_0-\frac12}^{\sqrt n x_0+\frac12}\log\left(1-\frac{(u-\sqrt n x_0)+(v-\sqrt n y_0)}{h_\square}\right)du\,dv,
\end{align*}
where $u=\sqrt n x$ and $v=\sqrt n y$.
Now expand the Taylor series of the logarithm and integrate term by term.
\end{proof}

It is now convenient to change variables to $s$ and $t$ such that $x=\frac1{\sqrt 2}(L(s)-s)$ and $y=\frac{1}{\sqrt 2}(L(t)+t)$. To explain how this change of variables works we have the following diagram.
\begin{center}
\includegraphics{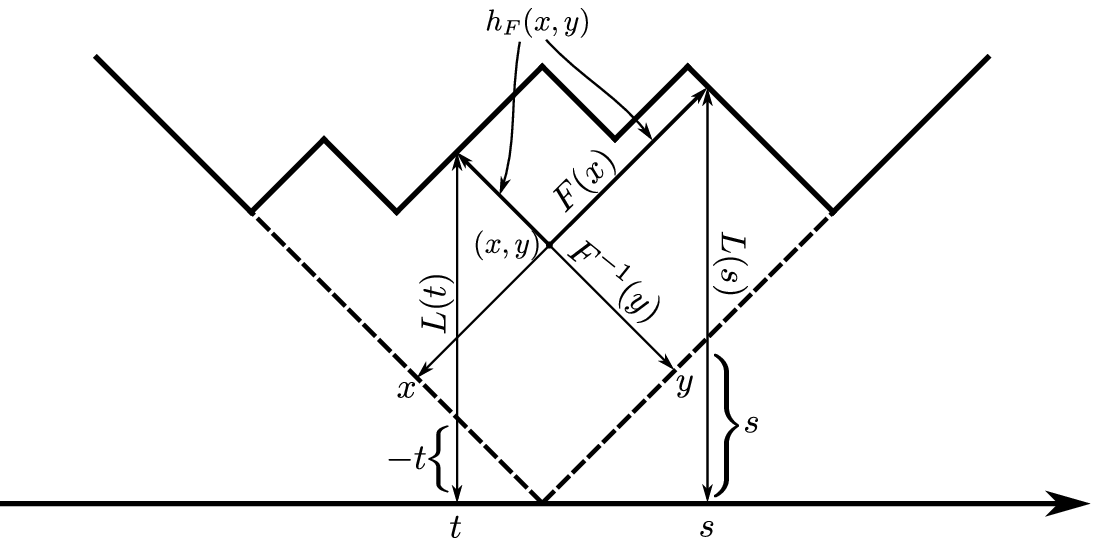}
\end{center}
In the diagram, the point $(x,y)$ is at the center, and through it pass two diagonals of lengths $F(x)$ and $F^{-1}(y)$ intersecting the diagonal axes at coordinates $x$ and $y$, respectively. From their intersections with the rim of the rotated Young diagram, we have vertical segments of lengths $L(s)$ and $L(t)$.
The transform $(s,t)\mapsto (x,y)$ is not bijective, but it is surjective, and it is injective in the region in which its Jacobian is nonzero. The determinant of its Jacobian is
 \[-\frac12\left(1-L'(s)\right)\left(1+L'(t)\right),\]
which is supported in a compact set because $L(x)=|x|$ for large $|x|$.
Our integral becomes, after a couple of integrations by parts:
\begin{multline}
\log \prod_{\square\in\lambda}h_\square=\\
\boxed{\frac n2\int\!\int_{s>t}\left(\log \sqrt 2(s-t)\right)(1-L'(s))(1+L'(t)) \,ds\,dt+\sum_{\square\in \lambda}c(h_\square)}
\label{eq:hookintegral}
\end{multline}
This integral is, in fact, the square of the norm of a Sobolev space, whence it is convex and has good properties for optimization.
Finally, we have:
\begin{lem}\label{lem:asymptoticsofremainder}
\[\sum_{\square\in\lambda} c(h_\square)=O\left(\sqrt{n}\right)\quad \textrm{as $n=|\lambda|\to\infty.$}\]
\end{lem}
\begin{proof}
Since $c$ is a decreasing function and $c(h)\to0$ very quickly as $h\to\infty$, the worst case scenario is the case in which we maximize the number of \emph{small} hook lengths $h_\square$. This happens in the case of the \emph{staircase partition} $(\ell,\ell-1,\dots,2,1)$, and in that case the estimate is obvious.
\end{proof}

\section{Probability distributions and limit shapes}
\label{sec:probability}

\subsection{The uniform distribution}
\label{sec:uniform}
Let $\Young_n$ be the set of Young diagrams corresponding to all partitions of $n$, rescaled by $1/\sqrt n$. This means that we construct the diagrams with squares of side $1/\sqrt n$.

It is a very interesting phenomenon that, for large $n$, most of the diagrams in $\Young_n$ are very similar. In fact, once rescaled, most of them are very close to the curve
\begin{equation}\label{eq:limitshapeuniform}
e^{-\frac\pi{\sqrt6}x}+e^{-\frac\pi{\sqrt6}y}=1.
\end{equation}
Here is a simulation:

\begin{center}
\includegraphics{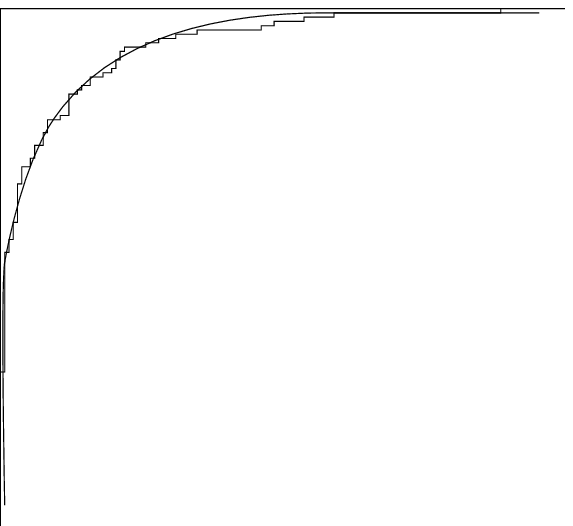}
\end{center}

In the picture we see the graph of the curve corresponding to the limit shape, together with the contour of a random partition of very large size $n$ rescaled by $1/\sqrt n$. The $x$-axis is horizontal and increases to the right; the $y$-axis is vertical and increases downwards.

This phenomenon was apparently first remarked by A. Vershik \cite{vershik96}, who mentioned it informally to the Szalay and Turan after they had published a paper \cite{szalayturan} from where it can be deduced immediately. 

Now, let us explain how one can arrive at this conclusion. We will do this in two ways.
First, let us apply the variational method explained in Section \ref{sec:entropy}. One does standard variational calculus (see for example \cite{gelfandfomin}) to find that the (Euler-Lagrange) equation  that must be satisfied by any function $f$ maximizing the entropy integral \eqref{eq:entropyintegral} is
\[f''=1-(f')^2.\]
It is easy to see that, after a rescaling to ensure that the limit shape has area 1, the only possible candidate is then \eqref{eq:limitshapeuniform}. To complete the argument, one must prove that the entropy integral \eqref{eq:entropyintegral} is upper-semicontinuous and from there obtain a large deviation principle. We will not go into the details, but they can be found for example in \cite{okounkovlimitshapes}.

Our second way to deduce the form of the limit shape leverages the tools we developed in Sections \ref{sec:shiftedsymmetric}, \ref{sec:fockspace} and \ref{sec:traces}. This will seem like a huge roundabout, but it is perhaps the simplest illustration of how one can deal with more complicated problems. The idea is similar to the one used to prove Wigner's semicircle law: we will find the moments of the limit shape, and since they turn out to coincide with those of our candidate \eqref{eq:limitshapeuniform}, we can apply a standard result \cite[Chapter 30]{billingsley} in probability that says that two densities must coincide if they have the same moments $\mu_k$ and these satisfy the growth condition
\[\limsup_{k\to\infty}\frac1{2k}\mu_{2k}^{1/2k}<\infty.\]

Instead of looking at each of the layers $\mathcal Y_n$ of partitions Young diagrams of partitions of $n$ separately, we introduce the \emph{$q$-coupled probability distribution} $m_q$ on the full set of Young diagrams $\bigcup_n \mathcal Y_n$. It will be given by
\[m_q(\lambda)=\frac{q^{|\lambda|}}{Z},\quad
Z=\sum_\nu q^{|\nu|}=\prod_{n=1}^\infty(1-q^n)^{-1}\]
for\footnote{We can only interpret these as probabilities if $q\in \R$, but it is useful to keep in mind the complex-analytic properties of the functions involved.} $|q|<1$, $q\in \C$.
 In the denominator we have normalized using the generating function $Z$ of partitions, whose product form is easy to verify.
In statistical mechanics the process of taking into account all possible states of all possible energies simultaneously is called `passing to the grand canonical ensemble.' In this regime, the limit $q\nearrow1$ corresponds to the limit $|\lambda|\to\infty$, as we will see below.

We introduce the \emph{1-point function}
\[F(e^r)=\frac1r+\sum_{j\geq0}\frac{r^j}{j!}\langle\p_j\rangle,\]
where the expectations of the shifted-symmetric power functions $\p_j$ are given by
\[\langle\p_j\rangle=\sum_{\lambda}\p_j(\lambda)\frac{q^{|\lambda|}}Z.\]
Because of the convergence of the expectations of the shifted-symmetric power functions $\p_k(\lambda)$ to the moments of the limit shape, as explained in Section \ref{sec:shiftedsymmetric}, the asymptotics of $F$ encode the moments of the limit shape (which we aim to compute).

There is another useful interpretation for $F$. Denote by $\mathfrak s(\lambda)$ the set of shifted coordinates $\lambda_i-i+\frac12$ of $\lambda$.
As we rescale the partition by $1/\delta$ and let $|\lambda|\to\infty$, we expect to have $f'(x)$ close to 1 in the regions where there is a high probability of finding one of these rescaled shifted coordinates $\delta(\lambda_i-i+\frac12)$ (because they correspond to the presence of a black pebble), and we expect $f'(x)$ to be close to $-1$ in the regions where this probability is very low. In other words,
\[f'(x)\approx\lim_{\delta\to0}\mathbf P(x/\delta\in \mathfrak s(\lambda)).\]
If we take the generating function of these probabilities,
\[\sum_{x\in\Z+\frac12}u^x\sum_{\mathfrak s(\lambda)\ni x} \frac{q^{|\lambda|}}Z,\]
where the second sum is taken over all partitions $\lambda$ such that $x$ is in the set of shifted coordinates $\mathfrak s(\lambda)$.
By equation \eqref{eq:generatingshiftedsympower}, this is the same as $F(u)$, where $u=e^r$.

Using this interpretation and the language of the half-infinite wedge, we rewrite $F(u)$ as a trace,
\[
F(u)=\frac1Z\sum_{x\in\Z+\frac12}u^x\trace q^H\psi_{ x}\psi^*_{x}=\frac1Z[v^0]\trace q^H\psi(uv)\psi^*(v).
\]
Our notations are as in equation \eqref{eq:bosongf}, and $[v^0]$ stands for the operation of ``taking the constant term in the series expansion with respect to $v$.'' The idea is that the composition of operators $\psi_x\psi^*_x$ will select the vectors $v_\lambda$ such that $x$ is in $\mathfrak s(\lambda)$. Apply formulas \eqref{eq:bosongf}, \eqref{eq:tracetoprod} and \eqref{eq:specialtrace} to find that
\begin{align*}
\frac1Z\trace q^H\psi(uv)\psi^*(v)&=\frac{1}{x^{1/2}-x^{-1/2}}\exp\sum_{n=1}^\infty
\frac{\left((uv)^n-v^n\right)\left(v^{-n}-(uv)^{-n}\right)q^n}{n(1-q^n)}\\
&=\frac{1}{x^{1/2}-x^{-1/2}}\exp\sum_{n=1}^\infty\left(u^n+u^{-n}-2\right)
\frac{\sum_{k=1}^\infty q^{nk}}{n}\\
&=\frac{1}{x^{1/2}-x^{-1/2}}\prod_{k=1}^\infty
\frac{(1-q^k)^2}{(1-q^ku)(1-q^k/u)}=\frac1{\vartheta(u)},
\end{align*}
where $\vartheta$ is the Jacobi theta function defined in equation \eqref{eq:jacobithetadef} and we used that
\[\log(1-a)=-\sum_{n=1}^\infty \frac{a^n}n.\]
So we have
\[F(u)=\frac1{\vartheta(u)}.\]

Using the same techniques as above, it is an easy exercise to see that, in the case of the uniform distribution, the generating functions of $n$-point correlations are given by
\[\sum_{i_1,i_2,\dots,i_n\in\Z+\frac12}u_1^{i_1}\cdots u^{i_n}_n\mathbf P_{\mu_q}\left(i_1,\dots,i_n\in\mathfrak s(\lambda)\right)=\prod_{k=1}^n\frac1{\vartheta(u_k)}\]
(where $\mathbf P_{\mu_q}$ denotes the probability with respect to the distribution $\mu_q$ that assigns weights $\mu_q(\lambda)=q^{|\lambda|}$),
and this is what is called the \emph{$n$-point function}. This is great news, because the fact that it factors as a product of thetas implies all sorts of good quasimodularity properties; in particular, it means that the $\mu_q$-expectations of the functions $\p_k$ (whose highest degree terms encode the moments of the limit shape as $q\to1$) are quasimodular forms. (Compare with Section \ref{sec:quasimodular}.)

We turn to the issue of understanding the asymptotic behavior of this distribution. As explained above, we want to be able to extract information about the moments of the limit shape from the $n$-point function.

In order to get an idea of how the moments will scale as $q\to1$, it is insightful to compute the leading term of the asymptotics of the expectation of the size of the partition,
\[\left\langle|\lambda|\right\rangle=\frac{\sum_\lambda |\lambda|q^{|\lambda|}}Z= \frac{qZ'}Z.\]
Recall that if $q=e^{2\pi i\tau}$, then the Dedekind eta function $\eta$ satisfies
\[\eta(\tau)=\sqrt\frac i\tau\eta\left(-\frac1\tau\right)=q^{1/24}\sum_{n\in\Z}(-q)^nq^{\frac{3n^2-n}2}
=q^{1/24}\prod_{n=1}^\infty(1-q^n).\]
Hence, as $\tau\to 0+0i$,
\begin{multline*}
Z=\prod_{n=1}^\infty(1-q^n)^{-1}=\frac{q^{1/24}}{\eta(q)}=\frac{q^{1/24}}{\sqrt{\frac i\tau}\eta\left(-\frac1\tau\right)}=\\
e^{\frac{\pi i \tau}{12}+\frac{\pi i}{12\tau}}\left(\sum_n(-1)^n e^{-\frac{\pi i}\tau(3n^2-n)}\right)^{-1}\sqrt{\frac\tau i}
\approx e^{\frac{\pi i}{12\tau}}\sqrt{\frac\tau i}.
\end{multline*}
Plugging this approximation into the expectation, we see its highest order asymptotics are given by
\[\frac{qZ'}{Z}\approx\frac{4\pi}{\tau^2}.\]
This is the expected size of the Young diagram of a partition when $q$ is close to 1. If we want to rescale by $1/\delta$ so that $\delta^2|\lambda|=1$, we see that we must keep $\delta\approx \tau/i$ (up to a real constant).

To respect this scaling, we will apply Lemma \eqref{lem:theta1} to the 1-point function rescaling also the variables, as follows:
\[
F\left(e^{\tau r/i}\right)=\frac{1}{\vartheta(e^{\tau r/i},e^{2\pi i \tau})}
\approx\frac{1}{2i \tau \sin(r/2)}\exp\left(\frac{\tau r^2}{4\pi}\right)
\]
The exponential term clearly tends to 1 as $\tau\to0$. The series of the cosecant encodes exactly the moments of \eqref{eq:limitshapeuniform} (up to a constant due to rescaling). Let us verify this. To compute the moments, change variables $x=s+t$ and $y=s-t$ and take the same integral as in equation \eqref{eq:momentintegral}:
\begin{align*}
\mu_{2k}&=\int_{\R}t^{2k}\left(L_\infty(t)-|t|\right)dt\\
&=
\int_{\R}t^{2k}
\left(\tfrac{\sqrt6}\pi\log\left[2\cosh \left(\tfrac{\pi}{\sqrt 6\,}t\right)\right]-|t|\right)dt\\
&=\frac{6^{k+1}}{2^{2k} \pi^{2k+2}}\int_{\R_+}t^{2k}\log\left(1+e^{-t}\right)dt\\
&=\frac{3^{k+1}}{2^{k-1}\pi^{2n+2}}\left(1-2^{-2n-1}\right)\Gamma(2n+1)\zeta(2n+2)
\end{align*}
by the Mellin transform (compare \cite[page 99]{branchedcoverings}), and $\mu_n=0$ for odd $n$. One can check that the even terms in the series expansion of $1/2\sin(r/2)$ are
\[\frac{1}{(2k-1)!}\frac{1}{(2\pi)^{2k+2}}\left(1-{2^{-2k-1}}\right)\Gamma(2k+1)\zeta(2k+2),\]
and the difference between these two things can be traced down to the scaling, but we will not do it here.

The existence of the limit shape also follows as a special case of the Large Deviation Principle proved by A. Dembo, A. Vershik and O. Zeitouni \cite{dembovershikzeitouni}.

\subsubsection*{Typical dimension.}
In order to investigate what the dimension is for a typical partition in the uniform distribution as $|\lambda|\to\infty$, we will use the hook formula \eqref{eq:hookformula}, Stirling's approximation for $n!$, and the approximation given by the hook integral \eqref{eq:hookintegral}:
\begin{align*}
\log\dim \lambda&=\log\frac{n!}{\prod_{\square\in\lambda}h_\square}\\
&\approx n\log n-n-\frac n2\int\!\int_{s>t}\log\sqrt 2(s-t)(1-L_\infty'(s))(1+L_\infty'(t))\,ds\,dt\\
&\approx n\log n-n\,0.628699\dots
\end{align*}

\subsection{The Plancherel distribution}
\label{sec:plancherel}
The \emph{Plancherel distribution} seems to have first appeared in the search for the asymptotics of the maximal dimension of an irreducible representation of $S(n)$. It assigns the weight
\[p(\lambda)=\frac{(\dim\lambda)^2}{|\lambda|!}=\frac{|\lambda|!}{ \left(\prod_{\square \in\lambda}h_\square\right)^2}\]
to the partition $\lambda$. (The equivalence of these two definitions follows from equation \eqref{eq:hookformula}.) By Burnside's formula,
\[n!=\sum_{|\lambda|=n}\left(\dim\lambda\right)^2,\]
the weights $p(\lambda)$ define a probability on the set $\mathcal Y_n$ of Young diagrams of partitions of $n$. It was named ``Plancherel'' because the Fourier transform $\mathcal F$, understood as a map
\[\mathcal F:L^2\left(S(n)^{S(n)},\mu_{\mathrm{Haar}}\right)\to L^2\left(\widehat{S(n)},\mu_\textrm{Plancherel}\right)\]
(where $S(n)^{S(n)}$ is the set of conjugacy classes of $S(n)$, and $\widehat{S(n)}\cong \mathcal Y_n$ denotes the space of complex finite-dimensional irreducible representations) is an isometry, like in the classical Plancherel theorem.

The variational methods explained in Section \ref{sec:hooksvariational} were originally developed for the study of this distribution \cite{loganshepp,vershikkerov1985}.
The maximizer of the resulting integral \eqref{eq:hookintegral} is the curve
\[L_\infty(x)=\left\{\begin{array}{ll}
\frac{2}\pi\left(x\arcsin x+\sqrt{1-x^2}\right)& \textrm{for $|x|\leq 1$,}\\
|x|& \textrm{for $|x|\geq 1$.}
\end{array}\right.
\]
For any other curve $\Omega$,  let $h(x)=\Omega(x)-L_\infty(x)$ be the variation. We will assume that it is compactly supported. In this case one has that the hook integral \eqref{eq:hookintegral} has the form
\begin{multline*}
-\frac n2\int\!\!\int_{\R^2}\log\sqrt2|s-t|h'(s)h'(t)\,ds\,dt\\
+\frac n2\int\!\!\int_{|s|,|t|>1} \log\sqrt2|s-t|\left(h'(s)-h'(t)\right)ds\,dt\\
=-\frac{n}{4}\int\!\!\int_{\R^2}\left(\frac{h(s)-h(t)}{s-t}\right)^2ds\,dt.
\end{multline*}
The second term in the left hand side disappears because of symmetry. To obtain the right hand side form the first term, we basically do two integrations by parts. We thus obtain a constant times the square of what actually is the norm in a Sobolev space, and it follows that the maximum is attained when $h=0$. The full central limit theorem is attributed to S. Kerov, but its proof was first written up by V. Ivanov and G. Olshanski \cite{ivanovolshanski}.

While the author is not aware of any methods to efficiently deal with this distribution using the half-infinite wedge fermionic Fock space, many interesting facts are known about this distribution. There are for example some papers that analyze very carefully the local correlations (e.g., \cite{borodinokounkovolshanski,bufetov}), and the papers dealing with the resolution of the Baik-Deift-Johansson conjecture \cite{baikdeiftjohansson} are very interesting (e.g., \cite{borodinokounkovolshanski,johansson,okorandommatrices}).

\subsubsection*{Typical dimension.} In the case of the Plancherel distribution, the typical dimension turns out to be the same as the maximal dimension \cite{vershikkerov1985}.

\subsection{The pillowcase distribution}
\label{sec:pillowcase}
The \emph{pillowcase distribution} initially appeared in \cite{pillow} in connection with the computation of the volumes of the moduli spaces of quadratic differentials, and was further studied in \cite{mythesis,mypillowcase}. Given a parameter $q$, $0<|q|<1$, it assigns to each partition $\lambda$ the weight $q^{|\lambda|}\w(\lambda)/Z$, where\footnote{ The original definition was given in \cite{pillow} in terms of characters of near-involutions. The equivalence is given in \cite{mythesis,mypillowcase}.}
\[\w(\lambda)=\left(\frac{\prod\textrm{odd hook lengths of $\lambda$}}{\prod\textrm{even hook lengths of $\lambda$}}\right)^2\]
to partitions $\lambda$ with empty 2-core, and no weight to any other partitions. The normalization constant $Z$ is known to be
\[Z=\prod_{n=1}^\infty(1-q^{2n})^{-1/2}.\]

In the case of the pillowcase distribution, one can apply the variational methods of Section \ref{sec:hooksvariational} as follows: Since the relevant partitions $\lambda$ have empty 2-core, their 2-quotient $(\alpha,\beta)$ determines them completely. According to the picture drawn in Section \ref{sec:pquotients}, the odd hook lengths of $\lambda$ correspond to twice the hook lengths of $\alpha$ and $\beta$, while the odd hooks of $\lambda$ correspond to the interactions of the Maya diagrams of $\alpha$ and $\beta$: there will be an odd hook whenever there is a pebble on the Maya diagram of $\alpha$ (resp. $\beta$) with a corresponding hole on the Maya diagram of $\beta$ (resp. $\alpha$), to the right of it.
Since the Maya diagrams of $\alpha$ and $\beta$ are interlaced to form the Maya diagram of $\lambda$, their contours interlace too.
Letting $L_\alpha$ and $L_\beta$ be the contours of the partitions $\alpha$ and $\beta$,  we have
\[L'_\lambda(s)=\left\{\begin{array}{ll}
L'_\alpha\left(s-\frac{k}{2\sqrt{2n}}\right), &
\textrm{for $s\in\frac1{\sqrt{2n}}(k-1,k),k\in2\Z,$}\\
L'_\beta\left(s-\frac{k}{2\sqrt{2n}}\right), &
\textrm{for $s\in\frac1{\sqrt{2n}}(k,k+1),k\in2\Z,$}
\end{array}\right.\]
Going through the derivation of equation \eqref{eq:momentintegral} again, one sees that the integral for the part corresponding to the even hooks is
\begin{multline*}
-n\!\int_{t<s}\!\!\left(\log\sqrt2(s-t)\right)\!\big(\!\left(1-L'_\alpha(s)\right)\!
\left(1+L'_\alpha(t)\right)
+\left(1-L'_\beta(s)\right)\!\left(1+L'_\beta(t)\right)\!\big)
ds\,dt,
\end{multline*}
and the one for the odd hooks is
\begin{multline*}
n\int_{t<s}\left(\log\sqrt2(s-t)\right)\!\big(\!\left(1-L'_\alpha(s)\right)\!
\left(1+L'_\beta(t)\right)
+\left(1-L'_\beta(s)\right)\!\left(1+L'_\alpha(t)\right)\!\big)
ds\,dt.
\end{multline*}
Adding and simplifying, one finds
\[ n\int_{\R^2}\left(\log\sqrt2|s-t|\right)\Delta'(s)\Delta'(t)\,ds\,d\]
where $\Delta(s)=L_\alpha(s)-L_\beta(s)$. After two intergrations by parts,
\[\log\w(\lambda)\approx t=-\frac n2\int_{\R^2}\left(\frac{\Delta(s)-\Delta(t)}{s-t}\right)^2ds\,dt+O\left(\sqrt n\right).\]
This integral is again a multiple of the square of the norm $\|\Delta\|^2$ of $\Delta$ in a Sobolev space. Whence one gets a hint that, in the limit, the pillowcase weights induce a distribution concentrated mostly near the partitions for which $\Delta=0$, that is, for which the 2-quotients are equal, and that in that region the distribution is more or less uniform. So one expects to get the same limit shape as for the uniform distribution, and one can indeed prove some statements in this direction \cite{mypillowcase,mythesis}. However, the term $O\left(\sqrt n\right)$ makes this conclusion hard to obtain using this method.

\subsubsection*{Fock space method.} It was found in \cite{pillow} that the \emph{pillowcase operator}, given by
\[\mathfrak W=\exp\left(\sum_{n>0}\frac{\alpha_{-2n-1}}{2n+1}\right)
\exp\left(-\sum_{n>0}\frac{\alpha_{2n+1}}{2n+1}\right),\]
has the property that its diagonal elements exactly coincide with the pillowcase weights:
\[\left\langle\mathfrak Wv_\lambda,v_\lambda\right\rangle=\w(\lambda).\]
One can thus proceed in the same way as we did in Section \ref{sec:uniform}: the 1-point function becomes
\[F\left(u\right)=\frac1Z[v^0]\trace q^H\mathfrak W\psi(uv)\psi^*(v).\]
Turning the crank as in Section \ref{sec:uniform}, one sees that
\[F\left(u\right) = \frac{1}{\vartheta(u)}[v^0]\sqrt{\frac{\vartheta(-v)\vartheta(uv)}
{\vartheta(v)\vartheta(-uv)}}.
\]
Similarly, the $n$-point function also has a closed expression in terms of modular forms, and again we have good quasimodularity properties. Applying Lemmas \ref{lem:theta1} and \ref{lem:theta2}, one obtains the desired result: essentially, in the $q\to1$ limit, the quotient inside the square root tends to 1 and the limit shape indeed coincides with the one for the uniform distribution. The typical dimension thus also coincides with the one for the uniform distribution.


\bibliography{bib}{}
\bibliographystyle{plain}

\end{document}